\documentclass[twoside,11pt]{article}

\usepackage{jmlr2e}
\usepackage{graphicx}
\usepackage{psfrag}
\usepackage{caption}
\usepackage{subfig}
\usepackage{url}
\usepackage{epsfig}
\usepackage{natbib}
\usepackage{color}
\usepackage{amsmath,amssymb}
\usepackage{multirow}

\newtheorem{assumption}[theorem]{Assumption}

\newcommand{\dist}{{\rm dist}}
\newcommand{\rset}{\mathbb{R}}

\newcommand{\prox}{\ensuremath{\operatorname{prox}}}

\newcommand{\red}{\textcolor{black}}
\newcommand{\blue}{\textcolor{black}}

\usepackage{lastpage}
\jmlrheading{23}{2022}{1-\pageref{LastPage}}{9/21; Revised
	4/22}{9/22}{21-1062}{Ion Necoara, Nitesh Kumar Singh}
\ShortHeadings{title}{Necoara, Singh}	

\ShortHeadings{Stochastic subgradient for composite  convex optimization with functional constraints}{ Necoara, Singh}

\begin{document}

\title{Stochastic subgradient projection methods for composite  optimization with  functional constraints}

\author{\name Ion Necoara  \email ion.necoara@upb.ro \\
       \addr Automatic Control and Systems Engineering Department,  University Politehnica  Bucharest, Spl. Independentei 313, 060042 Bucharest, Romania. \\ 
       Gheorghe Mihoc-Caius Iacob Institute of Mathematical Statistics and Applied Mathematics of the  Romanian Academy, 050711 Bucharest, Romania. 
       \AND
       \name Nitesh Kumar Singh \email nitesh.nitesh@stud.acs.upb.ro \\
       \addr Automatic Control and Systems Engineering Department,  University Politehnica  Bucharest,   Spl. Independentei 313, 060042 Bucharest, Romania. 
}

\editor{Silvia Villa}

\maketitle

\begin{abstract} %
In this paper we consider convex optimization problems with  stochastic composite objective function subject to (possibly) infinite intersection of constraints. The objective function is   expressed in terms of expectation operator over a sum of two terms   satisfying a stochastic bounded gradient condition, with or without strong convexity type properties. In contrast to the classical approach, where the constraints are usually represented as intersection of simple sets, in this paper we consider that each constraint set is  given as the level set of a convex but not necessarily differentiable function.  Based on the flexibility offered by our general optimization model we consider a stochastic subgradient method with random  feasibility updates.  At each iteration, our algorithm takes a stochastic proximal (sub)gradient step aimed at  minimizing the objective function and then a subsequent subgradient step  minimizing the feasibility violation of the observed random constraint.  We analyze the convergence behavior of the proposed algorithm  for  diminishing stepsizes and for the case when the objective function is  convex or \red{strongly convex}, unifying the nonsmooth and smooth cases.  We prove sublinear convergence rates for this stochastic subgradient algorithm, which are known to be optimal for subgradient  methods on  this class of problems. When the objective function has a linear least-square form and the constraints are polyhedral, it is shown that the algorithm converges linearly.   Numerical evidence supports the effectiveness of our method in real problems.
\end{abstract}


\section{Introduction}
\noindent The large sum of functions in the objective function and/or the large  number of constraints in  most of the practical optimization applications,  including   machine learning and statistics \citep{Vap:98,BhaGra:04}, signal processing \citep{Nec:20,Tib:11}, computer science \citep{KunBac:18},   distributed control \citep{NedNec:14}, operations research and finance \citep{RocUry:00},  create several theoretical and computational challenges.   For example, these problems are becoming increasingly large in terms of both the number of variables and the size of training data and they  are  usually nonsmooth due to the use of regularizers, penalty terms and the presence of constraints. Due to these challenges, (sub)gradient-based  methods are widely applied.  In particular, proximal gradient algorithms \citep{Nec:20,RosVil:14} are natural in applications where the function to be minimized is in composite form, i.e. the sum of a smooth  term and a nonsmooth regularizer (e.g.,  the indicator function of some simple constraints), as  e.g. the empirical risk in machine learning. In these algorithms, at each iteration the proximal operator defined by the nonsmooth component  is applied to the gradient descent step for the smooth data fitting component.    In practice, it is important to consider situations where these operations cannot be performed exactly. For example, the case where the gradient, the proximal operator or the projection can be computed only up to an error have been considered  in  \citep{DevGli:14,NedNec:14,PatNec:18,RasCha:20}, while  the situation  when only stochastic estimates of these operators   are available have been studied in \citep{RosVil:14,HarSin:16,PatNec:17,HerStu:20}. This latter setting, which is also of interest here,  is very important  in machine learning, where we have to minimize an expected  objective function with or without constraints from random samples \citep{BhaGra:04}, or  in statistics, where we need to minimize a finite sum objective  subject to functional constraints   \citep{Tib:11}.   
 
 \medskip 
 
\noindent   \textit{Previous work}. A very popular approach for minimizing an  expected or finite sum objective function is the stochastic gradient descent (SGD) \citep{RobMon:51, NemYud:83, HarSin:16} or the stochastic proximal point (SPP) algorithms \citep{MouBac:11, NemJud:09, Nec:20, PatNec:17, RosVil:14}.  In these  studies sublinear convergence is derived for SGD or SPP with decreasing  stepsizes  under the assumptions that  the objective function is smooth and (strongly) convex.   For nonsmooth stochastic convex optimization one can recognize two main approaches. The first one uses  stochastic variants of  the subgradient method combined with different averaging techniques, see e.g.   \citep{NemJud:09, YanLin:16}. The second line of research   is based on stochastic variants of  the proximal gradient method under the assumption that  the expected objective  function is in composite form  \citep{DucSin:09, Nec:20, RosVil:14}.  For both approaches, using decreasing  stepsizes,  sublinear convergence rates are derived under the assumptions that  the objective function is (strongly) convex.  Even though SGD and SPP are well-developed  methodologies, they only apply to  problems with simple constraints, requiring the whole feasible set  to be  projectable.  

\medskip 

\noindent  In spite of its wide applicability, the study on efficient solution methods for optimization problems with many constraints  is still limited.  The most prominent line of research in this area is the alternating projections, which focus  on applying random projections for solving problems that involve intersection of a (infinite) number of sets. The case when the objective function is not present in the formulation,  which corresponds to the convex feasibility problem,    stochastic alternating projection algorithms were analyzed  e.g., in~\citep{BauBor:96},  with linear convergence rates, provided that the sets satisfy some linear regularity condition.  Stochastic forward-backward algorithms have been also applied to solve optimization problems with many constraints. However, the papers introducing those general algorithms focus on proving only asymptotic convergence results without rates, or they assume the number of constraints is finite, which is more restricted than our settings, see e.g. \citep{BiaHac:17, Xu:18, WanChe:15}. In the case where the number of constraints  is finite and the objective function is deterministic, Nesterov's smoothing framework is studied in~\citep{TraFer:18} under the setting of accelerated proximal gradient methods.   Incremental subgradient  methods or primal-dual approaches were also proposed for solving   problems with finite intersection of simple sets through an exact penalty reformulation  in ~\citep{Ber:11,KunBac:18}.  

\medskip 

\noindent  The papers most related to our work are \citep{Ned:11, AngNec:19},  where  subgradient methods with random feasibility steps are  proposed for solving convex problems with \textit{deterministic} objective and many functional constraints.  However, the optimization problem, the algorithm  and consequently the convergence analysis are different from the present paper.  In particular, our algorithm is a \textit{stochastic proximal gradient}  extension of the algorithms proposed in~\citep{Ned:11, AngNec:19}.  Additionally, the stepsizes  in~\citep{Ned:11, AngNec:19} are chosen decreasing, while in the present work for strongly like  objective functions we derive  insightful stepsize-switching rules which describe when one should switch from a constant to a decreasing stepsize regime.  	Furthermore,  in~\citep{Ned:11} and  \citep{AngNec:19} sublinear convergence rates are established either   for  convex or strongly convex deterministic objective functions, respectively,  while in this paper  we prove (sub)linear rates under  an expected composite  objective function   which is either convex or strongly convex.   Moreover, \citep{Ned:11, AngNec:19} present separately the convergence analysis for smooth and nonsmooth objective, while in this paper we present a unified convergence analysis covering both cases through the so-called stochastic bounded gradient condition.  Hence, since we deal with  stochastic composite objective functions, smooth or nonsmooth, convex or strongly convex,  and since we consider a stochastic proximal gradient with new stepsize rules, our convergence analysis  requires additional insights that differ from that of~\citep{Ned:11,AngNec:19}.

\medskip 

\noindent In~\citep{PatNec:17} a  stochastic optimization problem with infinite intersection of sets is considered and  stochastic proximal point steps are combined  with alternating projections for solving it.  However,  in order to prove sublinear convergence rates, \citep{PatNec:17} requires   strongly convex and smooth objective functions,  while our results are valid also for convex non-smooth functions. Lastly,~\citep{PatNec:17} assumes the projectability of individual sets, whereas in our case, the constraints might not be projectable.  Finally, in all these studies the nonsmooth component is assumed to be proximal, i.e. one can easily compute  its proximal operator. This assumption is restrictive, since in many applications the nonsmooth term is also expressed as expectation or finite sum of nonsmooth functions, which individually are proximal, but it is hard to compute the proximal operator for the  whole nonsmooth component, as e.g.,  in  support vector machine or generalized lasso problems \citep{VilRos:14,RosVil:14}. Moreover, all the convergence results from  the previous papers are derived  separately  for the smooth and the nonsmooth stochastic optimization problems.  

\medskip 

\noindent   \textit{Contributions}.   In this paper we remove the previous  drawbacks. We propose  a stochastic subgradient algorithm for solving general convex optimization problems having the objective function expressed in terms of expectation operator over a sum of two terms  subject to (possibly infinite number of) functional constraints.  The only assumption we require is to have access to an unbiased estimate of the (sub)gradient and of the proximal operator of  each of these two terms and to the subgradients of the constraints. To deal with such    problems, we propose  a stochastic  subgradient  method  with random feasibility updates.  At each iteration, the algorithm takes a stochastic subgradient step aimed at only minimizing the expected objective function, followed by a feasibility step  for  minimizing the feasibility violation of the observed random constraint achieved through Polyak's subgradient iteration, see \citep{Pol:69}.  In doing so, we can avoid the need for projections onto the whole set of  constraints, which may be  expensive computationally.    The proposed algorithm is applicable to the situation where the whole objective function and/or constraint set are not known in advance, but they are rather learned in time through observations.  

\medskip 

\noindent We present  a general framework for the convergence  analysis of  this stochastic subgradient  algorithm which  is based on the assumptions that   the objective function satisfies a stochastic bounded gradient condition,  with or without strong convexity and the subgradients of the functional constraints are bounded.  These assumptions  include the most well-known  classes of objective functions and of constraints analyzed in the literature:  composition of a  nonsmooth function  and a smooth function, with or without strong convexity, and nonsmooth Lipschitz  functions, respectively.  Moreover, when the objective function satisfies  the strong convexity property, we prove insightful stepsize-switching rules which describe when one should switch from a constant to a decreasing stepsize regime.  Then, we prove sublinear convergence rates  for the weighted averages of the iterates  in terms of expected distance to the constraint set, as well as for the expected optimality of the function values/distance to the optimal set.  Under some special conditions we also derive linear rates.   Our convergence  estimates   are known to be optimal for this class of stochastic subgradient schemes  for solving (nonsmooth) convex problems with functional constraints.    Besides providing a general framework for the design and analysis of   stochastic  subgradient methods, in special cases, where complexity bounds are known for some particular problems, our convergence results recover the  existing bounds.  In particular,  for problems without functional constraints we recover the  complexity bounds from \citep{Nec:20} and for linearly constrained least-squares problems we get  similar  convergence bounds  as in  \citep{LevLew:10}. 

\medskip 

\noindent   \textit{Content}.  In Section 2 we present our optimization problem and the main assumptions.  In Section 3 we design a stochastic subgradient projection algorithm and analyze its convergence properties, while is Section 4 we adapt this algorithm to constrained least-squares problems and derive linear convergence. Finally, in Section 5 detailed numerical simulations are provided that support the effectiveness of our method in real problems.


\section{Problem formulation and assumptions}
We consider the general convex composite  optimization problem with expected objective function and functional constraints:
\begin{equation}
	\label{eq:prob}
	\begin{array}{rl}
		F^* = & \min\limits_{x \in \mathcal{Y} \subseteq \mathbb{R}^n} \; F(x)  \quad \left(:= \mathbb{E}[f(x,\zeta) + g(x,\zeta)]\right)\\
		& \text{subject to } \;  h(x,\xi) \le 0 \;\; \forall \xi \in \Omega_2,
	\end{array}
\end{equation}
where the composite objective function $F$ has a stochastic representation in the form of expectation w.r.t.  a random variable $\zeta\in\Omega_1$, i.e., $\mathbb{E} [f(x,\zeta)+g(x,\zeta)]$,  $\Omega_2$ is an arbitrary collection of indices and $\mathcal{Y}$ is a simple closed convex set.  \blue{Hence, we separate the feasible set in two parts:  one set, $\mathcal{Y}$, admits easy projections and the other part is not easy for projection as it is  described by the level sets of  some convex functions $h(\cdot,\xi)$'s.}  Moreover,  $f(\cdot,\zeta), g(\cdot,\zeta)$ and $h(\cdot,\xi)$ are  proper lower-semicontinuous convex functions \blue{ containing  the interior of $\mathcal{Y}$  in their domains}.   One can notice that more commonly, one sees a single $g$ representing the regularizer on the parameters in the formulation \eqref{eq:prob}. However, there are also applications where one encounters terms of the form $\mathbb{E}[g(x,\zeta)]$ or $\sum_{\xi \in \Omega_2} g(x,\xi)$, as e.g.,  in Lasso problems with mixed $\ell_1-\ell_2$ regularizers or regularizers with overlapping groups, see e.g.,  \citep{VilRos:14}. Multiple functional constraints, onto which is difficult to project,  can arise from robust classification in machine learning,  chance constrained problems, min-max games and control, see  \citep{BhaGra:04,RocUry:00,PatNec:17}.  For further use, we define the following functions: $F(x,\zeta) = f(x,\zeta)+g(x,\zeta)$, $f(x) = \mathbb{E} [f(x,\zeta) ]$ and $g(x) = \mathbb{E} [g(x,\zeta)]$.  Moreover,   $\mathcal{Y}$ is assumed to be simple, i.e. one can easily compute the projection of a point onto this set. Let us define the individual sets $\mathcal{X}_\xi$ as $\mathcal{X}_\xi=\left\{ x \in  \mathbb{R}^n: \;  h(x,\xi)\le 0 \right\}$ for all $\xi \in \Omega_2.$ Denote the feasible set of \eqref{eq:prob} by:  
$$\mathcal{X}=\left\{ x \in  \mathcal{Y}: \;  h(x,\xi)\le 0 \;\;  \forall \xi \in \Omega_2 \right\} =  \mathcal{Y} \cap \left(\cap_{\xi \in \Omega_2} \mathcal{X}_\xi \right). $$ We  assume $\mathcal{X}$ to be nonempty.  We also assume that the optimization problem \eqref{eq:prob} has finite optimum and we let  $F^*$ and $\mathcal{X}^*$ denote the optimal value and the optimal set, respectively:
\[F^*=\min_{x  \in \mathcal{X}} F(x) := \mathbb{E}[F(x,\zeta)], \quad \mathcal{X}^*=\{x\in \mathcal{X} \mid F(x)=F^*\} \not= \emptyset.\]
Further, for any $x \in \rset^n$ we denote its projection onto the optimal set  $\mathcal{X}^*$ by $\bar{x}$, that is:  $$\bar{x} = \Pi_{\mathcal{X}^*}(x).$$ 
For the objective function we  assume that the first term $f(\cdot,\zeta)$ is either differentiable or nondifferentiable function and we use,  with some abuse of notation, the same notation for the gradient or the subgradient of $f(\cdot,\zeta)$ at $x$, that is $\nabla f(x, \zeta) \in \partial f(x, \zeta)$, where the subdifferential $\partial f(x, \zeta)$ is either a singleton or a nonempty  set for any $\zeta \in \Omega_1$. The other term $g(\cdot,\zeta)$  is assumed to have an easy proximal operator for any $\zeta \in \Omega_1$:
\[   \prox_{\gamma g(\cdot,\zeta)}(x) = \arg \min\limits_{y \in \rset^m} g(y,\zeta) + \frac{1}{2 \gamma} \|y
-x\|^2. \]

\noindent Recall that  the proximal operator of the indicator function of a closed convex set  becomes the projection. We consider additionally the following assumptions on  the objective function and constraints:
\begin{assumption}
	\label{assumption1}
	The (sub)gradients of $F$ satisfy a stochastic bounded gradient condition, that is there exist non-negative constants  $L \geq 0$ and $B \geq 0$ such that:
	\begin{equation}
		\label{as:main1_spg} 
		B^2 +  L(F(x) - F^*) \geq \mathbb{E}[\| \nabla F(x, \zeta) \|^2 ] \quad  \forall x \in  \mathcal{Y}.
	\end{equation}
\end{assumption}

\noindent \red{From Jensen's inequality, taking $x=x^* \in \mathcal{X}^*$ in  \eqref{as:main1_spg}, we get:
\begin{equation}
	\label{as:main1_spg2} 
	B^2  \geq \mathbb{E}[\| \nabla F(x^*, \zeta) \|^2 ] \ge \| \mathbb{E}[\nabla F(x^*, \zeta)] \|^2 = \| \nabla F(x^*)\|^2 \quad \forall  x^* \in \mathcal{X}^*.
\end{equation}
We also assume $F$ to satisfy \red{a (strong) convexity} condition: 
\begin{assumption}
	\label{assumption2}
	The function  $F$ satisfies a \red{(strong) convex} condition on $\mathcal{Y}$, i.e., there exists non-negative constant $\mu \geq 0$ such that:
	\begin{equation}
		\label{as:strong_spg} 
		F(y) \geq F(x) + \langle \nabla F(x), y-x \rangle +  \frac{\mu}{2}  \|y  - x\|^2 \quad \forall x, y \in \mathcal{Y}.
	\end{equation}
\end{assumption}
}

\noindent  Note that when $\mu=0$ relation 	\eqref{as:strong_spg} states that $F$ is convex on $\mathcal{Y}$. Finally, we assume boundedness on the subgradients of  the functional constraints:
\begin{assumption}
	\label{assumption3}
	The functional constraints $h(\cdot,\xi)$  have  bounded subgradients on  $ \mathcal{Y}$, i.e., there exists non-negative constant $B_h>0$ such that for all  $\nabla h(x,\xi)  \in \partial h(x,\xi)$, we have: 
	$$ \|  \nabla h(x,\xi) \| \le B_h \quad   \forall  x  \in   \blue{\mathcal{Y}}  \;\;  \text{and } \;\;  \xi \in \Omega_2. $$
\end{assumption}  

\noindent Note that our assumptions are quite general and cover the most well-known  classes of functions analyzed in the literature. In particular,   Assumptions 	\ref{assumption1} and \ref{assumption2},  related to the objective function,  covers the class of non-smooth Lipschitz  functions and composition of a (potentially) non-smooth function  and a smooth function, with or without strong convexity, as the following examples show. 

\medskip 

\noindent \textbf{Example 1} [Non-smooth (Lipschitz) functions satisfy Assumption \ref{assumption1}]:
Assume that the functions $f$ and $g$ have bounded
(sub)gradients:
\[ \|\nabla f (x, \zeta)\| \leq B_f  \quad \text{and} \quad
\|\nabla g(x,\zeta) \| \leq B_g \quad \forall x  \in \mathcal{Y}.  \]
Then, obviously Assumption \ref{assumption1} holds with $ L=0 \quad  \text{and} \quad B^2 = 2 B_f^2
+ 2 B_g^2.$

\medskip 

\noindent \textbf{Example 2} [Smooth (Lipschitz gradient) functions satisfy Assumption \ref{assumption1}]:  Condition  \eqref{as:main1_spg}  includes the class of  functions formed as a sum of two terms,  $f(\cdot, \zeta)$  \red{convex} having Lipschitz continuous gradient  and $g(\cdot, \zeta)$   \red{convex} having bounded subgradients, and  $\mathcal{Y}$ bounded. 
Indeed, let us assume that the convex function  $f(\cdot, \zeta)$ has  Lipschitz continuous 
gradient, i.e. there exists $L_f(\zeta)>0$ such that:
\[  \|\nabla f(x, \zeta) - \nabla f(\bar{x}, \zeta)\| \leq L_f(\zeta) \|x - \bar{x}\| \quad \forall x \in \mathcal{Y}. \]
Using standard arguments (see Theorem 2.1.5 in \citep{Nes:18}), we have:
\[  f(x, \zeta) -  f(\bar{x}, \zeta) \geq  \langle \nabla f(\bar{x},\zeta), x -\bar{x} \rangle    + \frac{1}{2L_f(\zeta)} \|\nabla f(x, \zeta) - \nabla f(\bar{x}, \zeta) \|^2. \]
Since  $g(\cdot, \zeta)$ is also convex, then adding  $g(x, \zeta) - g(\bar{x}, \zeta) \geq \langle \nabla g(\bar{x}, \zeta), x -\bar{x} \rangle $ in the previous inequality, where  $\nabla g(\bar{x}, \zeta) \in \partial g(\bar{x}, \zeta)$, we get:
\begin{align*}
F(x,\zeta) -  F(\bar{x},\zeta)  & \geq  \langle \nabla F(\bar{x}, \zeta), x -\bar{x} \rangle    + \frac{1}{2L_f(\zeta)} \|\nabla f(x, \zeta) - \nabla f(\bar{x}, \zeta) \|^2,
\end{align*}
where we used that $\nabla F(\bar{x}, \zeta) = \nabla f(\bar{x}, \zeta) + \nabla g(\bar{x}, \zeta) \in \partial F(\bar{x}, \zeta)$.  Taking expectation w.r.t. $\zeta$ and assuming that the set $\mathcal{Y}$ is bounded, with the diameter $D$, then after using Cauchy-Schwartz inequality,   we get (we also assume  $L_f(\zeta) \leq L_f$ for all $\zeta$):
\begin{align*}
F(x) - F^*  & \geq \frac{1}{2L_f} \mathbb{E} [\|\nabla f(x, \zeta) - \nabla f(\bar{x}, \zeta) \|^2] - D\|\nabla F(\bar{x})\| \quad \forall x \in \mathcal{Y}.
\end{align*}
Therefore, for any $\nabla g(x, \zeta) \in \partial g(x, \zeta)$, we have:
\begin{align*}
\mathbb{E} [\|\nabla F(x, \zeta) \|^2] & =   \mathbb{E} [\|\nabla f(x, \zeta) - \nabla f(\bar{x}, \zeta) + \nabla g(x, \zeta) + \nabla f(\bar{x}, \zeta) \|^2] \\
& \leq 2\mathbb{E}[ \|\nabla f(x, \zeta) - \nabla f(\bar{x}, \zeta)\|^2]  +2\mathbb{E}[ \|\nabla g(x, \zeta) + \nabla f(\bar{x}, \zeta)\|^2] \\
&\leq 4L_f (F(x) - F^*) +  4L_f D \|\nabla F(\bar{x})\| +2\mathbb{E}[ \|\nabla g(x, \zeta)  + \nabla f(\bar{x}, \zeta)\|^2].
\end{align*}
Assuming now  that  the regularization functions $g$ have bounded
subgradients on $\mathcal{Y}$, i.e., $\|\nabla g(x, \zeta)\| \leq B_g$, then  the  bounded gradient condition  \eqref{as:main1_spg} holds  on on $\mathcal{Y}$ with:
\[  L =  4 L_f \quad \text{and} \quad B^2 =  4  \left(B_g^2 +\max_{\bar{x} \in \mathcal{X}^*} \left(\mathbb{E}[ \|\nabla f(\bar{x}, \zeta)\|^2] +  D L_f \|\nabla F(\bar{x})\| \right) \right). \]
\noindent  Note that $B^2$ is finite, since  the optimal set $\mathcal{X}^*$ is compact (recall that in this example $\mathcal{Y}$ is assumed bounded).

\medskip 

\noindent Finally, we also impose some regularity condition  for the constraints.
\begin{assumption}
	\label{assumption4}
	The functional constraints satisfy a regularity condition, i.e., there exists  non-negative constant $c>0$ such that:
	\begin{align}
		\label{eq:constrainterrbound}
		\dist^2(x, \mathcal{X}) \le c \cdot  \mathbb{E} \left[ (h(x,\xi))_+^2 \right] \;\; \forall x \in \blue{\mathcal{Y}}. 
	\end{align}
\end{assumption}

\noindent Note that this assumption  holds e.g., when the index set $\Omega_2$ is arbitrary and the feasible set $\mathcal{X}$ has an interior point, see   \citep{Pol:01}, or when the feasible set is polyhedral,  see  relation (\ref{hoffman}) below. However, Assumption (\ref{assumption4}) holds  for more general sets, e.g., when a strengthened Slater condition holds for the collection of convex functional constraints, such as the generalized Robinson condition, as detailed in Corollary 3 of \citep{LewPan:98}. 

\medskip


\section{Stochastic  subgradient projection algorithm }
Given the iteration counter $k$, we consider independent  random variables $ \zeta_k$ and $\xi_k$  sampled from $\Omega_1$ and $\Omega_2$ according to probability distributions $\textbf{P}_1$ and $\textbf{P}_2$, respectively.  Then, we define the following  stochastic subgradient projection algorithm, where at each iteration we perform  a stochastic proximal (sub)gradient step aimed at  minimizing the expected composite  objective function and then a subsequent subgradient step  minimizing the feasibility violation of the observed random constraint (we use the convention  $0/0=0$)\footnote{In this  variant we have corrected some derivations from the   journal version:  Journal of Machine Learning Research, 23(265), 1-35, 2022.}:
\begin{center}
	\noindent\fbox{%
		\parbox{12cm}{%
			\textbf{Algorithm 1 (SSP)}:\\
			$\text{Choose} \; x_0 \in \mathcal{Y} \; \text{and stepsizes} \; \alpha_k>0 $  and $\beta \in (0, 2)$\\
			$\text{For} \; k \geq 0 \;  \text{repeat:}$
			\begin{align}
				& \text{Sample independently} \;   \zeta_k \sim \textbf{P}_1 \;\text{and}\; \xi_k\sim \textbf{P}_2  \; \text{and update:} \nonumber  \\
				&\blue{u_k} = \text{prox}_{\alpha_k g(\cdot,\zeta_{k})} \left(x_{k} - \alpha_{k} \nabla f(x_{k}, \zeta_{k} )\right), \quad v_k = \blue{\Pi_\mathcal{Y} (u_k)} \label{eq:algstep1}\\
				&z_k = v_k  - \beta \frac{(h(v_k,\xi_k))_+}{\| \nabla h(v_k,\xi_k) \|^2} \nabla h(v_k,\xi_k) \label{eq:algstep2}\\
				&x_{k+1} = \Pi_{\mathcal{Y}} (z_k).\label{eq:algstep3}
			\end{align}
		}%
	}
\end{center}

\noindent Here, $\alpha_{k}>0$ and $\beta>0$ are deterministic stepsizes and $(x)_+ = \max \{0,x\}$. Note that $v_k$  represents a stochastic proximal (sub)gradient step (or a stochastic  forward-backward step) at $x_k$ for the expected  composite objective function $F(x) = \mathbb{E}[f(x,\zeta) + g(x,\zeta)]$.  Note that the optimality step selects a random pair of functions  $(f(x_k,\zeta_k), g(x_k,\zeta_k))$ from the composite objective function $F$  according to the probability distribution $\textbf{P}_1$, i.e., the index variable $\zeta_k$ with values in the set $\Omega_1$.  Also, we note that the random feasibility step selects a random constraint $ h(\cdot,\xi_k) \leq 0$ from the collection of constraints set according to the probability distribution $\textbf{P}_2$, independently from $\zeta_k$, i.e. the index variable $\xi_k$ with values in the set $\Omega_2$. The vector $\nabla h(v_k, \xi_k)$ is chosen as: 
\[
\nabla h(v_k, \xi_k)  = 
\begin{cases} 
\nabla h(v_k, \xi_k) \in \partial ((h(v_k,\xi_k))_+)  & \mbox{if }  (h(v_k,\xi_k))_+ > 0  \\ 
	s_h  \neq 0 & \mbox{if } (h(v_k,\xi_k))_+ = 0, 
\end{cases}
\] 
where $s_h \in \mathbb{R}^n$ is any nonzero vector.   If $(h(v_k,\xi_k))_+ = 0$,  then  for any choice of nonzero  $s_h$, we have $z_k = v_k$. Note that the feasibility step \eqref{eq:algstep2} has the special form of the Polyak's subgradient iteration, see e.g.,  \citep{Pol:69}.  Moreover, when $\beta = 1$, $z_k$ is the projection of  $v_k$ onto the  hyperplane: 
$$\mathcal{H}_{v_k,\xi_k} =\{ z: \;  h(v_k,\xi_k) + \nabla h(v_k,\xi_k)^T(z - v_k) \leq 0  \}, $$ that is  $z_k = \Pi_{\mathcal{H}_{v_k,\xi_k}} (v_k)$ for $\beta = 1$.   Indeed,  if $v_k \in \mathcal{H}_{v_k, \xi_k}$, then  $(h(v_k, \xi_{k}))_+ = 0$ and the projection is the point itself, i.e., $z_k = v_k$.  On the other hand, if $v_k \notin \mathcal{H}_{v_k, \xi_k}$, then the projection of $v_k$ onto $\mathcal{H}_{v_k, \xi_k}$ reduces to the projection onto the corresponding hyperplane:
\begin{align*}
	z_k = v_k  - \frac{h(v_k,\xi_k)}{\| \nabla h(v_k,\xi_k) \|^2} \nabla h(v_k,\xi_k).
\end{align*} 
Combining these two cases, we finally get our update \eqref{eq:algstep2} .  Note that when the feasible set of optimization  problem \eqref{eq:prob} is described by (infinite) intersection of simple convex sets, see e.g. \citep{PatNec:17,BiaHac:17, Xu:18, WanChe:15}: 
$$  \mathcal{X} =  \mathcal{Y}  \cap  \left(\cap_{\xi  \in \Omega_2} \mathcal{X}_\xi  \right),   $$
where each set $X_\xi$ admits an easy  projection, then one can choose the following functional representation  in problem \eqref{eq:prob}:  
$$ h(x,\xi) = (h(x,\xi))_+ =  \dist (x, \mathcal{X}_\xi)      \quad \forall \xi \in \Omega_2.   $$ 
 One can easily notice that this function is convex, nonsmooth and with bounded subgradients, since we have \citep{MorNam:05}: 
\[   \frac{ x - \Pi_{\mathcal{X}_\xi}(x) }{\|  x - \Pi_{\mathcal{X}_\xi}(x)  \|} \in \partial  h(x,\xi).   \]
In this case,  step \eqref{eq:algstep2} in SSP algorithm becomes a usual  random projection step:
\[ z_k = v_k - \beta(v_k  - \Pi_{\mathcal{X}_{\xi_{k}}}(v_k)).   \]  
Hence, our formulation is more general than \citep{PatNec:17,BiaHac:17, Xu:18, WanChe:15}, as it allows to also deal  with constraints that might not be projectable, but one can easily compute a subgradient of $h$.   We mention also the possibility of performing  iterations in parallel in steps \eqref{eq:algstep1} and \eqref{eq:algstep2} of SSP algorithm. However, here we do not consider this case. A thorough convergence analysis of minibatch iterations when the objective function is deterministic and strongly convex can be found in \citep{AngNec:19}.   For the analysis of SSP algorithm, let us define the stochastic (sub)gradient mapping (for simplicity we omit its dependence on stepsize~$\alpha$):
\[  \mathcal{S}(x, \zeta) = \alpha^{-1} \left( x - \prox_{\alpha g(\cdot,\zeta)}(x - \alpha {\nabla} f(x, \zeta)) \right). \] Then, it follows immediately that   the  stochastic proximal (sub)gradient step \eqref{eq:algstep1} can  be written as:
\[ \blue{u_{k}} = x_k - \alpha_k \mathcal{S}(x_k, \zeta_k).  \]
Moreover, from  the optimality condition of the prox operator it follows that  there exists
$\nabla g(\blue{u_{k}},\zeta_{k}) \in \partial g(\blue{u_{k}},\zeta_k)$ such that:
\[  \mathcal{S}(x_k, \zeta_k) =  \nabla f(x_k, \zeta_k)  + \nabla g(\blue{u_{k}},\zeta_{k}).  \]
Let us also recall a basic property of the projection onto a  closed convex set $\mathcal{X} \subseteq \mathbb{R}^n$, see e.g.,  \citep{AngNec:19}:
\begin{align}\label{proj_prop}
	\|\Pi_{\mathcal{X}}(v) - y\|^2 \leq \|v - y\|^2 - \|\Pi_{\mathcal{X}} (v)-v\|^2 \qquad \forall  v \in \mathbb{R}^n\; \text{and}\; y \in \mathcal{X}. 
\end{align} 
Define also  the filtration: 
\[  \mathcal{F}_{[k]}= \{ \zeta_0,\cdots,\zeta_k,  \;\;  \xi_0,\cdots,\xi_k  \}.   \] 
The next lemma provides a key  descent property for  the sequence $v_k$ and for the  proof we use as main tool the stochastic  (sub)gradient mapping~$\mathcal{S}(\cdot)$.

\begin{lemma}
	\label{th:spg_basic} 
	Let  $f(\cdot,\zeta)$ and $g(\cdot,\zeta)$ be  convex functions.  Additionally,  assume that the bounded gradient condition from   Assumption \ref{assumption1} holds. Then, for any $k \geq 0$ and  stepsize $\alpha_k >0$, we have the following recursion:
	\begin{align}
		\label{spg_basic} 
		&\mathbb{E}[\|v_{k} - \bar{v}_{k} \|^2] 
		 \leq \mathbb{E}[\| x_k - \bar{x}_{k} \|^2] -  \alpha_k  (2 - \alpha_k L) \, \mathbb{E}[ F(x_k) - F(\bar{x}_{k}) ] + \alpha_k^2 B^2.   
	\end{align}
\end{lemma}

\begin{proof}
Recalling that $\bar{v}_k =\Pi_{\mathcal{X}^*}(v_k)$ and $\bar{x}_k = \Pi_{\mathcal{X}^*}(x_k)$,  from the definition of $\blue{u_k}$ and using \eqref{proj_prop} for $y = \bar{x}_k \in \mathcal{X^*} \blue{\subseteq \mathcal{Y}}$ and $v = v_k$, we get:
	\begin{align}
		\label{der_spg_basic}
		 \| v_{k} - \bar{v}_{k} \|^2 & \overset{\eqref{proj_prop}}{\le}   \| v_{k} - \bar{x}_{k} \|^2 \blue{ =\| \Pi_\mathcal{Y} (u_{k}) - \Pi_\mathcal{Y}( \bar{x}_{k}) \|^2 \le \| u_{k} - \bar{x}_{k} \|^2} =  \| x_{k} - \bar{x}_{k} - \alpha_k \mathcal{S}(x_k, \zeta_k)\|^2 \nonumber\\
		& =  \| x_{k} - \bar{x}_{k} \|^2  - 2 \alpha_k \langle \mathcal{S}(x_k, \zeta_k), x_k  - \bar{x}_{k} \rangle + \alpha_k^2 \| \mathcal{S}(x_k, \zeta_k) \|^2   \nonumber \\
		& = \| x_{k} - \bar{x}_{k} \|^2  - 2 \alpha_k \langle \nabla
		f(x_k, \zeta_k) + \nabla g(\blue{u_{k}},\zeta_{k}), x_k  - \bar{x}_{k} \rangle + \alpha_k^2 \| \mathcal{S}(x_k, \zeta_k) \|^2. 
	\end{align}
	Now, we refine the second term. First, from convexity of $f$ we
	have:
	\[ \langle \nabla f(x_k, \zeta_k), x_k  - \bar{x}_{k} \rangle \geq f(x_k, \zeta_k) - f(\bar{x}_{k}, \zeta_k). \]
	Then, from  convexity of $g(\cdot,\zeta)$ and the definition of  the 
	gradient mapping $\mathcal{S}(\cdot, \zeta)$, we have:
	\begin{align*}
		&\langle \nabla g(\blue{u_{k}},\zeta_{k}), x_k  - \bar{x}_{k} \rangle  = \langle  \nabla g(\blue{u_{k}},\zeta_{k}), x_k  - \blue{u_{k}} \rangle + \langle  \nabla g(\blue{u_{k}},\zeta_{k}), \blue{u_{k}} - \bar{x}_{k} \rangle \\
		& \geq \alpha_k \| \mathcal{S}(x_k, \zeta_k) \|^2 - \alpha_k \langle \nabla f(x_k, \zeta_k), \mathcal{S}(x_k, \zeta_k) \rangle + g(\blue{u_{k}},\zeta_{k}) - g(\bar{x}_k,\zeta_{k})\\
		& \geq \alpha_k \| \mathcal{S}(x_k, \zeta_k) \|^2 - \alpha_k \langle
		\nabla f(x_k, \zeta_k) + \nabla g(x_k, \zeta_{k}), \mathcal{S}(x_k, \zeta_k) \rangle  + g(x_{k},\zeta_{k}) - g(\bar{x}_k,\zeta_{k}).
	\end{align*}
	Replacing the previous two inequalities in \eqref{der_spg_basic}, we
	obtain:
	\begin{align*}
		\| v_{k} - \bar{v}_{k} \|^2 & \leq  \| x_{k} - \bar{x}_{k} \|^2
		- 2 \alpha_k  \left( f(x_k, \zeta_k) + g(x_k,\zeta_{k}) - f(\bar{x}_{k}, \zeta_k) - g(\bar{x}_{k},\zeta_{k})\right)  \\
		& \quad + 2 \alpha_k^2 \langle \nabla f(x_k, \zeta_k) + \nabla g(x_k,\zeta_{k}), \mathcal{S}(x_k, \zeta_k) \rangle - \alpha_k^2 \|
		\mathcal{S}(x_k, \zeta_k) \|^2.
	\end{align*}
Using that  $2\langle u , v \rangle - \|v\|^2 \leq \|u\|^2$ for
	all $v \in \rset^n$, that $ F(x_k, \zeta_k)  =  f(x_k, \zeta_k) +  g(x_k,\zeta_{k})$ and $\nabla F(x_k, \zeta_k)  = \nabla f(x_k, \zeta_k) + \nabla g(x_k,\zeta_{k}) \in \partial F(x_k, \zeta_k)$, we further get:
	\begin{align*}
		\| v_{k} - \bar{v}_{k} \|^2 & \leq \| x_{k} - \bar{x}_{k} \|^2
		- 2 \alpha_k \left(  F(x_k, \zeta_k)  - F(\bar{x}_{k}, \zeta_k) \right)  +  \alpha_k^2 \| \nabla F(x_k, \zeta_k)\|^2.
	\end{align*}
Since 	$v_{k}$ depends  on $\mathcal{F}_{[k-1]} \cup \{\zeta_k\}$, not on $\xi_k$, and the stepsize $\alpha_k$ does not  depend on $(\zeta_k,\xi_k)$, then from basic properties of conditional expectation, we have:
	\begin{align*}
		&\mathbb{E}_{ \zeta_{k} } [\| v_{k} - \bar{v}_{k} \|^2 | \mathcal{F}_{[k-1]} ] \\ 
		& \leq \| x_{k} - \bar{x}_{k} \|^2
		- 2 \alpha_k  \mathbb{E}_{\zeta_{k}}[  F(x_k, \zeta_k)  - F(\bar{x}_{k}, \zeta_k)   | \mathcal{F}_{[k-1]} ]  +  \alpha_k^2 \mathbb{E}_{\zeta_{k}}[\| \nabla F(x_k, \zeta_k)\|^2 | \mathcal{F}_{[k-1]} ] \\ 
		& = \| x_{k} - \bar{x}_{k} \|^2 - 2 \alpha_k  \left(  F(x_k)  - F(\bar{x}_{k}) \right) +  \alpha_k^2 \mathbb{E}_{\zeta_{k}}[\| \nabla F(x_k, \zeta_k)\|^2  | \mathcal{F}_{[k-1]}]  \\
		& \leq  \| x_{k} - \bar{x}_{k} \|^2 - 2 \alpha_k  \left(  F(x_k)  - F(\bar{x}_{k}) \right)  + \alpha_k^2 \left( B^2 + L \left( F(x_k) - F(\bar{x}_{k})\right) \right) \\
		& =  \|x_{k} - \bar{x}_{k} \|^2  - \alpha_k (2 - \alpha_k L)  \left( F(x_k) -
		F(\bar{x}_{k})\right) + \alpha_k^2 B^2,
	\end{align*}
where in the last inequality we used  $x_k \in \mathcal{Y}$ and the  stochastic bounded gradient inequality from  Assumption \ref{assumption1}.  	Now,  taking  expectation w.r.t.  $\mathcal{F}_{[k-1]} $ we get: 
	\begin{align*}
		&\mathbb{E}[\|v_{k} - \bar{v}_{k} \|^2]  \leq \mathbb{E}[\| x_k - \bar{x}_{k} \|^2] -  \alpha_k  (2 - \alpha_k L) \mathbb{E}[( F(x_k) - F(\bar{x}_{k}) )] + \alpha_k^2 B^2.
	\end{align*}
which concludes our statement. 	
\end{proof}

\vspace{-0.2cm}

\noindent Next lemma gives a relation between $x_k$ and $v_{k-1}$ (see also~\citep{AngNec:19}). 

\begin{lemma}
\label{lem:distanyy}
Let  $h(\cdot,\xi)$  be  convex functions.	Additionally, Assumption \ref{assumption3} holds.   Then,  for any $y \in  \mathcal{Y}$ such that $(h(y,\xi_{k-1}))_+ = 0$ the following relation holds:
	\begin{align*}
		\|x_{k}-y\|^2\le \|v_{k-1}-y\|^2-\beta(2-\beta) \left[ \frac{(h(v_{k-1},\xi_{k-1}))_+^2}{B^2_h} \right].
	\end{align*}
\end{lemma}

\begin{proof}
Consider any $y \in \mathcal{Y}$ such that $(h(y,\xi_{k-1}))_+ = 0$. Then,  using the nonexpansive  property of the projection and the definition of $z_{k-1}$, we have:
	\begin{align*}
		\|x_k - y\|^2& = \|\Pi_{\mathcal{Y}}(z_{k-1}) -y\|^2  \leq\|z_{k-1}-y\|^2\\
		& =\|v_{k-1} -y - \beta \frac{(h(v_{k-1},\xi_{k-1}))_+}{\|\nabla h(v_{k-1},\xi_{k-1})\|^2} \nabla h(v_{k-1},\xi_{k-1})\|^2\\
		& = \|v_{k-1}-y\|^2 + \beta^2 \frac{(h(v_{k-1},\xi_{k-1}))_+^2}{\|\nabla h(v_{k-1},\xi_{k-1})\|^2}\\
		& \quad -2 \beta \frac{(h(v_{k-1},\xi_{k-1}))_+}{\|\nabla h(v_{k-1},\xi_{k-1})\|^2}\langle v_{k-1}-y, \nabla h(v_{k-1},\xi_{k-1})\rangle\\
		& \leq \|v_{k-1}-y\|^2 + \beta^2 \frac{(h(v_{k-1},\xi_{k-1}))_+^2}{\|\nabla h(v_{k-1},\xi_{k-1})\|^2} -2 \beta \frac{(h(v_{k-1},\xi_{k-1}))_+}{\|\nabla h(v_{k-1},\xi_{k-1})\|^2} (h(v_{k-1},\xi_{k-1}))_+,
	\end{align*}
where the last inequality follows from  convexity of $(h(\cdot, \xi))_+$ and our assumption that $(h(y,\xi_{k-1}))_+ = 0$. After rearranging the terms and using that $v_{k-1} \in \mathcal{Y}$, we get:
	\begin{align*}
		\|x_k -y\|^2 &\leq \|v_{k-1}-y\|^2 - \beta(2-\beta)\left[ \frac{(h(v_{k-1},\xi_{k-1}))_+^2}{\|\nabla h(v_{k-1},\xi_{k-1})\|^2}\right] \\
		& \overset{\text{Assumption} \, \ref{assumption3}}{\leq} \|v_{k-1}-y\|^2 - \beta(2-\beta) \left[ \frac{(h(v_{k-1},\xi_{k-1}))_+^2}{B_h^2}\right],
	\end{align*}
which concludes our statement. 
\end{proof}
In the next sections, based on  the previous two lemmas,  we derive convergence rates for SSP algorithm depending on the (convexity) properties of $f(\cdot,\zeta)$.


\subsection{Convergence analysis: convex objective function}
\noindent In this section we consider that the functions $f(\cdot,\zeta), g(\cdot,\zeta)$ and $h(\cdot,\xi)$ are convex.  First, it is easy to prove that $c B_h^2 > 1$ (we can always choose $c$ sufficiently large such that this relation holds), see also \citep{AngNec:19}. For simplicity of the exposition let us introduce the following  notation: 
\[ C_{\beta,c,B_h} := \frac{ \beta(2-\beta)}{c B^2_h}   \left( 1 -   \frac{\beta(2-\beta) }{c B^2_h} \right)^{-1} >0.  \] 
\red{We impose the following conditions on the stepsize $\alpha_k$:
\begin{align}
\label{eq:alk}
0 < \alpha_k \leq \alpha_k(2-\alpha_k L) <1 \;\; \iff \;\; 	 \alpha_k \in 
\begin{cases}
	\left(0, \frac{1}{2} \right) \;\;  \text{if} \; L =0  \\
	\left(0, \frac{1- \sqrt{(1-L)_+}}{L} \right) \;\;  \text{if} \; L > 0.
\end{cases} 
\end{align}	 
}
Then, we can define the following  average sequence generated by the algorithm SSP: 
\[  \hat{x}_k = \frac{\sum_{j=1}^{k} \alpha_j \red{(2 -\alpha_j L)} x_j}{S_k}, \quad  \text{where}  \; S_k = \sum_{j=1}^{k} \alpha_j \red{(2 -\alpha_j L)}.   \] 
\red{Note that this type of average sequence is also consider in \citep{GarGow:23} for unconstrained stochastic optimization problems.} The next theorem derives sublinear  convergence rates for the average sequence  $\hat{x}_k$. 
\begin{theorem}
\label{th:nonstrongconv}
Let  $f(\cdot,\zeta), g(\cdot,\zeta)$ and  $h(\cdot,\xi)$  be  convex functions. Additionally,   Assumptions \ref{assumption1},  \ref{assumption3} and \ref{assumption4} hold.  Further, choose  the stepsize sequence  \red{$\alpha_k$  as in \eqref{eq:alk}}  and  stepsize  $\beta \in (0, 2)$. Then, we have the following estimates  for the average sequence  $\hat{x}_k$ in terms of optimality and feasibility violation for problem \eqref{eq:prob}: 
	\begin{align*}
		& \mathbb{E}\left[  F(\hat{x}_k ) - F^*  \right] \leq    \frac{\| v_0 - \overline{v}_{0} \|^2}{ S_k }  +  \frac{B^2 \sum_{j=1}^{k} \alpha_j^2 }{ S_k }, \\
		& \mathbb{E}\left[ \dist^2(\hat{x}_k, \mathcal{X}) \right]  \leq    \frac{ \| v_0 - \overline{v}_{0} \|^2}{ C_{\beta,c,B_h} \cdot  S_k }  +  \frac{ B^2  \sum_{j=1}^{k} \alpha_j^2}{C_{\beta,c,B_h} \cdot  S_k}.
	\end{align*}  
\end{theorem}

\begin{proof}
Recall that from Lemma \ref{th:spg_basic}, we have:
\begin{align}
	\label{eq:central}
& \mathbb{E}\left[\|v_{k} - \bar{v}_{k} \|^2\right]  \leq \mathbb{E}\left[\| x_k - \bar{x}_{k} \|^2\right] -  \alpha_k  (2 - \alpha_k L) \mathbb{E}\left[ F(x_k) - F(\bar{x}_{k}) \right] + \alpha_k^2 B^2
\end{align}
Now, for  $y=\bar{v}_{k-1} \in \mathcal{X}^* \subseteq \mathcal{X} \subseteq \mathcal{Y}$ we have that  $(h(\bar{v}_{k-1},\xi_{k-1}))_+ = 0$, and thus using Lemma~\ref{lem:distanyy},  we get: 
	\begin{align*}
		\|x_k - \bar{x}_k\|^2  \overset{\eqref{proj_prop}}{\le} \|x_{k}-\bar{v}_{k-1}\|^2 \leq \|v_{k-1} -\bar{v}_{k-1}\|^2 - \beta(2-\beta) \left[\frac{(h(v_{k-1},\xi_{k-1}))_+^2}{B^2_h}\right].
	\end{align*}
	Taking conditional expectation on  $\xi_{k-1}$ given  $\mathcal{F}_{[k-1]} $, we get:
	\begin{align*}
		\mathbb{E}_{\xi_{k-1}} [\|x_k - \bar{x}_k\|^2 |  \mathcal{F}_{[k-1]} ] &\leq \|v_{k-1} -\bar{v}_{k-1}\|^2 - \beta(2-\beta) \mathbb{E}_{\xi_{k-1}} \left[\frac{(h(v_{k-1},\xi_{k-1}))_+^2}{B^2_h}| \mathcal{F}_{[k-1]} \right]\\
		& \overset{\eqref{eq:constrainterrbound}}{\leq} \|v_{k-1} -\bar{v}_{k-1}\|^2 - \frac{ \beta(2-\beta)}{c B^2_h}  \dist^2(v_{k-1}, \mathcal{X}).
	\end{align*}
	Taking now full expectation, we obtain: 
	\begin{align}  
		\label{eq:yvk-1}
			\mathbb{E} [\|x_k - \bar{x}_k\|^2  ] \leq 	\mathbb{E} [ \|v_{k-1} -\bar{v}_{k-1}\|^2 ]  - \frac{ \beta(2-\beta)}{c B^2_h}  \mathbb{E} [ \dist^2(v_{k-1}, \mathcal{X})],  
\end{align}			
	and using this relation  in  	\eqref{eq:central}, we get:
	\begin{align}
		\label{eq:rec1}
		& \mathbb{E}\left[\|v_{k} - \bar{v}_{k} \|^2 \right]  + \frac{ \beta(2-\beta)}{c B^2_h} \mathbb{E}\left[ \dist^2(v_{k-1}, \mathcal{X}) \right]  +\alpha_k  (2 - \alpha_k L)  \mathbb{E}\left[ F(x_k) - F(\bar{x}_{k})  \right] \nonumber  \\ 
		& \leq   \mathbb{E}[\| v_{k-1} - \bar{v}_{k-1} \|^2]  + \alpha_k^2 B^2. 
	\end{align}
Similarly, for $y=\Pi_{\mathcal X} ({v}_{k-1}) \subseteq \mathcal{X} \subseteq \mathcal{Y}$ we have that $(h(\Pi_{\mathcal{X}}({v}_{k-1}),\xi_{k-1}))_+ = 0$, and thus using again  Lemma~\ref{lem:distanyy}, we obtain:
	\begin{align*}  
 & \dist^2(x_{k}, \mathcal{X})  = \|x_{k} - \Pi_{\mathcal X} ({x}_{k})\|^2 \le \|x_{k} - \Pi_{\mathcal X} ({v}_{k-1}) \|^2 \\  
 & \leq  \dist^2(v_{k-1}, \mathcal{X}) -\beta(2-\beta) \frac{(h(v_{k-1},\xi_{k-1}))_+^2}{B^2_h}. 
		  \end{align*}
Taking conditional expectation on  $\xi_{k-1}$ given  $\mathcal{F}_{[k-1]} $, we get:
	\begin{align*}    
	& \mathbb{E}_{\xi_{k-1}}\left[\dist^2(x_{k}, \mathcal{X})  |  \mathcal{F}_{[k-1]} \right] \\ & \leq  \dist^2(v_{k-1}, \mathcal{X})  -   \frac{\beta(2-\beta) }{B^2_h}      \mathbb{E}_{\xi_{k-1}} \left[   (h(v_{k-1},\xi_{k-1}))_+^2 | \mathcal{F}_{[k-1]} \right]\\
	& \overset{\eqref{eq:constrainterrbound}}{\leq}  	\left( 1 -   \frac{\beta(2-\beta) }{c B^2_h} \right)   \dist^2(v_{k-1}, \mathcal{X}).  
	\end{align*} 
	After taking full expectation, we get:
	\begin{align}
		\label{eq:distvdistx}    
		\mathbb{E}\left[   \dist^2(x_{k}, \mathcal{X}) \right] & \leq   
		\left( 1 -   \frac{\beta(2-\beta) }{c B^2_h} \right) \mathbb{E}\left[   \dist^2(v_{k-1}, \mathcal{X}) \right]. 
	\end{align}
	Using	\eqref{eq:distvdistx}   in \eqref{eq:rec1}, we obtain:
	\begin{align*}
	& \mathbb{E}\left[\|v_{k} - \bar{v}_{k} \|^2 \right]  + \frac{ \beta(2-\beta)}{c B^2_h}   \left( 1 -   \frac{\beta(2-\beta) }{c B^2_h} \right)^{-1}   \mathbb{E}\left[ \dist^2(x_{k}, \mathcal{X}) \right] \\
	& +\alpha_k  (2 - \alpha_k L)  \mathbb{E}\left[ F(x_k) - F(\bar{x}_{k})  \right]  \leq  \mathbb{E}\left[ \| v_{k-1} - \bar{v}_{k-1} \|^2 \right] + \alpha_k^2 B^2.\nonumber  
\end{align*}


\noindent \red{ Since  $\alpha_k$  satisfies \eqref{eq:alk}, then $\alpha_k(2-\alpha_kL)  \leq 1$, and we further get the following recurrence:  
\begin{align}
	\label{eq:rec2_2}
	& \mathbb{E}\left[\|v_{k} - \bar{v}_{k} \|^2 \right]  + C_{\beta,c,B_h}  \alpha_k  (2 - \alpha_k L)   \mathbb{E}\left[ \dist^2(x_{k}, \mathcal{X}) \right] \\
	& +\alpha_k  (2 - \alpha_k L)  \mathbb{E}\left[ F(x_k) - F(\bar{x}_{k})  \right]  \leq  \mathbb{E}\left[ \| v_{k-1} - \bar{v}_{k-1} \|^2 \right] + \alpha_k^2 B^2.\nonumber  
\end{align}   
	Summing  \eqref{eq:rec2_2} from $1$ to $k$, we get: 
	\begin{align*}
		& \mathbb{E}\left[\|v_{k} - \bar{v}_{k} \|^2 \right]  + C_{\beta,c,B_h} \sum_{j=1}^{k} \alpha_j  (2 - \alpha_j L) \mathbb{E}\left[ \dist^2(x_{j}, \mathcal{X}) \right] \\
		& + \sum_{j=1}^{k} \alpha_j  (2 - \alpha_j L)  \mathbb{E}\left[  F(x_j) - F^*  \right]  \leq   \| v_0 - \bar{v}_{0} \|^2  +  B^2 \sum_{j=1}^{k} \alpha_j^2.
	\end{align*}
Using the definition of the  average sequence $\hat{x}_k$ and the  convexity of $F$ and of $ \dist^2(\cdot, \mathcal{X})$, we further get sublinear rates in expectation for the average sequence in terms of  optimality: 
	\begin{align*}
		& \mathbb{E}\left[ F(\hat{x}_k ) - F^*  \right] \leq   \sum_{j=1}^{k} \frac{\alpha_j(2 - \alpha_j L)}{S_k}   \mathbb{E}\left[  F(x_j) - F^*  \right]    \leq   \frac{\| v_0 - \bar{v}_{0} \|^2}{S_k}  +  B^2 \frac{\sum_{j=1}^{k} \alpha_j^2}{S_k},
	\end{align*}  
	and  feasibility violation:
	\begin{align*}
	 C_{\beta,c,B_h}\mathbb{E}\left[ \dist^2(\hat{x}_k, \mathcal{X}) \right]  & \leq  C_{\beta,c,B_h} \sum_{j=1}^{k} \frac{\alpha_j(2 - \alpha_j L)}{S_k} \mathbb{E}\left[ \dist^2(x_{j}, \mathcal{X}) \right] \\ 
		& \leq   \frac{\| v_0 - \bar{v}_{0} \|^2}{S_k}  +  B^2 \frac{\sum_{j=1}^{k} \alpha_j^2}{S_k}.
	\end{align*}
These conclude  our statements.   }                                                
\end{proof}

\noindent Note that for stepsize $\alpha_k=\frac{\alpha_0}{(k+1)^\gamma}$, with  $\gamma \in [1/2, 1)$  \red{and $\alpha_0$ satisfies \eqref{eq:alk}}, we have:
\[  S_k \red{\overset{\eqref{eq:alk}}{\geq}}  \sum_{j=1}^{k} \alpha_j \geq {\cal O}(k^{1-\gamma}) \quad \text{and} \quad \sum_{j=1}^{k} \alpha_j^2 \leq 
\begin{cases}
	{\cal O}(1) \;\;  \text{if} \; \gamma>1/2 \\
	{\cal O}(\ln(k)) \;\;\;  \text{if} \; \gamma=1/2. 
\end{cases} 
\] 
Consequently for  $\gamma \in (1/2, 1)$ we obtain from Theorem \ref{th:nonstrongconv} the following sublinear convergence rates: 
\begin{align*}
	& \mathbb{E}\left[ ( F(\hat{x}_k ) - F^* ) \right] \leq    \frac{\| v_0 - \bar{v}_{0} \|^2}{\red{{\cal O}(k^{1-\gamma})}}  +  \frac{B^2\red{{\cal O}(1)}}{\red{{\cal O}(k^{1-\gamma})}}, \\
	& \mathbb{E}\left[ \dist^2(\hat{x}_k, \mathcal{X}) \right]  \leq    \frac{ \| v_0 - \bar{v}_{0} \|^2}{ C_{\beta,c,B_h} \cdot \red{{\cal O}(k^{1-\gamma})}}  +  \frac{ B^2\red{{\cal O}(1)}}{C_{\beta,c,B_h} \cdot \red{{\cal O}(k^{1-\gamma})}}.
\end{align*}   
\red{ For the particular choice $\gamma=1/2$, we get similar rates as above just replacing ${\cal O}(1)$ with $	{\cal O}(\ln(k))$. However,  if we neglect the logarithmic terms, we get  sublinear convergence rates of order:}
\begin{align*}
	\mathbb{E}\left[ ( F(\hat{x}_k ) - F^* ) \right] \leq    {\cal O} \left( \frac{1}{k^{1/2}} \right) \quad \text{and} \quad \mathbb{E}\left[ \dist^2(\hat{x}_k, \mathcal{X}) \right]  \leq {\cal O} \left( \frac{1}{k^{1/2}} \right).
\end{align*}

\noindent It is important to note that when $B=0$, from  Theorem \ref{th:nonstrongconv} improved rates can be derived for SSP algorithm in the convex case. More precisely,  for stepsize $\alpha_k=\frac{\alpha_0}{(k+1)^\gamma}$, with  $\gamma \in [0, 1)$  \red{and $\alpha_0$ satisfies \eqref{eq:alk}}, we obtain convergence rates for  $\hat{x}_k$  in optimality and feasibility violation of order  ${\cal O} \left( \frac{1}{k^{1-\gamma}} \right)$. In particular,  for $\gamma=0$ (i.e., constant stepsize $\alpha_k= \alpha_0 \in \left(0, \min(\frac{1}{2},\frac{1}{L}) \right)$ for all $k \geq 0$) the previous convergence estimates yield rates of order  ${\cal O} \left( \frac{1}{k} \right)$.  
From our best knowledge these rates are new for stochastic subgradient methods applied on the class of optimization problems \eqref{eq:prob}.


\subsection{Convergence analysis: \red{strongly convex} objective function} \label{sec3.2}
In this section we additionally assume the \red{strong convex} inequality from Assumption \ref{assumption2} with $\mu>0$.  The next lemma derives an improved recurrence relation for the sequence $v_k$  under  the \red{strong convexity. Furthermore, due to strongly convex assumption on $F$, problem \eqref{eq:prob} has a unique optimum, denoted $x^*$. Our proofs from this section are different from the ones  in \citep{AngNec:19}, since here we consider a more general type of bounded gradient condition, i.e., Assumption \ref{assumption1}.  }

\begin{lemma}
\label{lemma3.4}
Let  $f(\cdot,\zeta), g(\cdot,\zeta)$ and  $h(\cdot,\xi)$  be  convex functions.  Additionally,  Assumptions \ref{assumption1}--\ref{assumption4} hold, \red{with $\mu>0$}. Define \blue{$k_0 = \lfloor\frac{8L}{\mu} - 1 \rfloor$}, $\beta \in \left(0,2\right)$, $\theta_{L,\mu} \!=\! 1  \!-\! \mu/(4L)$  and $\alpha_k \!=\! \frac{4}{\mu} \gamma_{k}$,  where $\gamma_k$ is given~by:  
	\begin{equation*}
		\gamma_{k} = \left\{\begin{array}{ll}\frac{\mu}{4L} & \text{\emph{if}}\;\; k\leq k_0\\
			\frac{2}{k+1}& \text{\emph{if}}\;\; k > k_0.
		\end{array}
		\right.
	\end{equation*}
	Then, the iterates of SSP algorithm  satisfy the following recurrence:
	\begin{align*}
		& \mathbb{E}[\|v_{k_0} - x^*\|^2] 
		\leq    \left\{\begin{array}{ll}    \frac{B^2}{L^2}    & \text{\emph{if}} \;\; \theta_{L,\mu}  \leq 0\\
			\theta_{L,\mu}^{k_0}  \|v_{0} - x^*\|^2 +  \frac{1 - \theta_{L,\mu}^{k_0}  }{1 - \theta_{L,\mu}} \left( 1 + \frac{5}{2C_{\beta,c,B_h} \theta_{L,\mu}} \right)  \frac{B^2}{L^2}  & \text{\emph{if}} \;\; \theta_{L,\mu}  >  0, 
		\end{array}
		\right. \\
		&\mathbb{E}[\|v_k - x^*\|^2] +\gamma_{k} \mathbb{E}[\| x_k - x^*\|^2] + \frac{C_{\beta,c,B_h}}{6} \mathbb{E}[\dist^2(x_{k}, \mathcal{X})]\\
		&  \leq \left(1 - \gamma_{k}\right)\mathbb{E}[\|v_{k-1}- x^*\|^2] + \left( 1 + \frac{6}{C_{\beta,c,B_h} } \right)\frac{16 B^289}{\mu^2}\gamma_{k}^2   \qquad  \forall k> k_0.
	\end{align*}
\end{lemma}

\begin{proof}
One can easily see that our stepsize can be written equivalently as $\alpha_k = \min\left(\frac{1}{L}, \frac{8}{\mu (k + 1)}\right)$ (for $L=0$ we use the convention  $1/L=\infty$). 
Using  Assumption \ref{assumption2} in Lemma~\ref{th:spg_basic}, we get: 
	\begin{align}
		\label{ineq17}
		&\mathbb{E}[\|v_{k} - x^*\|^2]  \leq  \mathbb{E}[\| x_k - x^*\|^2] -  \alpha_k\red{ (2 - \alpha_k L)} \mathbb{E}[( F(x_k) - F^* )] + \alpha_k^2 B^2 \nonumber \\
		 & \overset{\text{Assumption} \, \ref{assumption2}}{\leq}   \mathbb{E}[\| x_k - x^*\|^2] - \red{\alpha_k (2 - \alpha_k L) \mathbb{E}\left[ \langle \nabla F(x^*), x_k - x^*\rangle + \frac{\mu}{2} \| x_k - x^*\|^2 \right] } + \alpha_k^2 B^2\nonumber\\
		 & \le \red{\left(1 - \frac{\mu \alpha_k}{2} \right)\mathbb{E}[\| x_k - x^*\|^2] - \alpha_k (2 - \alpha_k L) \mathbb{E}\left[ \langle \nabla F(x^*), x_k - \Pi_\mathcal{X} (x_k) \rangle \right] + \alpha_k^2 B^2} \nonumber\\
		 & \red{\le \left(1 - \frac{\mu \alpha_k}{2} \right)\mathbb{E}[\| x_k - x^*\|^2] + \frac{\eta}{2} \mathbb{E}[\|x_k - \Pi_\mathcal{X} (x_k)\|^2] +  \frac{\alpha^2_k (2 - \alpha_k L)^2}{2\eta} \mathbb{E}\left[ \| \nabla F(x^*)\|^2 \right] + \alpha_k^2 B^2}\nonumber\\
		 & \red{\overset{\eqref{as:main1_spg2}}{\le} \left(1 - \frac{\mu \alpha_k}{4} \right)\mathbb{E}[\| x_k - x^*\|^2] - \frac{\mu \alpha_k}{4} \mathbb{E}[\| x_k - x^*\|^2]  + \frac{\eta}{2} \mathbb{E}[\|x_k - \Pi_\mathcal{X} (x_k)\|^2]  \nonumber}\\
		 	& \quad \red{+  \left( 1 + \frac{2}{\eta} \right) \alpha_k^2 B^2 },
	\end{align}
\red{where the third inequality uses the property of the stepsize (if  $\alpha_{k} \leq \frac{1}{L}$, then $(2-\alpha_{k}L) \geq 1$),  together with the optimality condition ($ \langle \nabla F(x^*), \Pi_\mathcal{X} (x_k) - x^*\rangle \ge 0$ for all $k$), the fourth inequality uses the identity $\langle a, b\rangle \ge -\frac{1}{2\eta} \|a\|^2 - \frac{\eta}{2} \|b\|^2$ for any $a,b \in \mathbb{R}^n, \eta >0$, and the final inequality uses the fact that $(2 - \alpha_k L)^2  \le 4$. }
From 	\eqref{eq:yvk-1}, we also have: 
	\begin{align}\label{rel2}
		& \mathbb{E}[\|x_k - x^*\|^2]  \overset{\eqref{eq:yvk-1}}{\leq} \mathbb{E}[\| v_{k-1} - x^*\|^2] - \beta(2-\beta)\frac{\mathbb{E} [\dist^2(v_{k-1}, \mathcal{X})]}{c B^2_h}\nonumber\\
		& \overset{(\ref{eq:distvdistx})}{\leq} \mathbb{E}[\| v_{k-1} - x^*\|^2] - \frac{\beta(2-\beta)}{cB_h^2}\left(1 - \frac{\beta(2-\beta)}{cB_h^2}\right)^{-1}\mathbb{E}[\dist^2(x_{k}, \mathcal{X})]\nonumber\\
		& = \mathbb{E}[\| v_{k-1} - x^*\|^2] - C_{\beta,c,B_h}\mathbb{E}[\dist^2(x_{k}, \mathcal{X})].
	\end{align}
\red{Taking $\eta = C_{\beta,c,B_h}\left(1 - \frac{\mu \alpha_k}{4} \right)>0$ and combining \eqref{ineq17} with \eqref{rel2},  we obtain:
\begin{align}
	\label{eq:1way}
	\mathbb{E}[\|v_{k} - x^*\|^2] \le & \left(1 - \frac{\mu \alpha_k}{4} \right) \mathbb{E}[\| v_{k-1} - x^*\|^2] - \frac{1}{2} C_{\beta,c,B_h}\left(1 - \frac{\mu \alpha_k}{4} \right) \mathbb{E}[\dist^2(x_{k}, \mathcal{X})] \nonumber \\
	& - \frac{\mu \alpha_k}{4} \mathbb{E}[\| x_k - x^*\|^2] +  \left( 1 + \frac{2}{C_{\beta,c,B_h} \left(1 - \frac{\mu \alpha_k}{4} \right)} \right) \alpha_k^2 B^2.
\end{align}
}
For $L=0$ we have $k_0=0$. For $L>0$, then $k_0>0$ and  for any $k \leq k_0$, we have $\alpha_{k} = \frac{1}{L}$. Hence,   from \eqref{eq:1way},  we obtain  for any $k \leq k_0$:
	\begin{align*}
		\mathbb{E}[\|v_k - x^*\|^2]
		& \leq \left(  1 - \frac{\mu}{4L}  \right) \mathbb{E}[\|v_{k-1} - x^*\|^2]  + \left( 1 + \frac{2}{C_{\beta,c,B_h} \left(1 - \frac{\mu}{4L}\right)} \right) \frac{B^2}{L^2}\\
		&  \leq \max \left(  \left(1 - \frac{\mu}{4L}  \right) \mathbb{E}[\| v_{k-1} - x^*\|^2]  + \left( 1 + \frac{2}{C_{\beta,c,B_h} \left(1 - \frac{\mu}{4L}\right)} \right)\frac{B^2}{L^2}, \frac{B^2}{L^2} \right).
	\end{align*}
Using the geometric sum formula and recalling that $\theta_{L,\mu} = 1- \mu/(4L)$,  we obtain the first statement.   
\noindent Further,  for $k>k_0$, from relation \eqref{eq:1way}, we have: 
	\begin{align*}
		&\mathbb{E}[\|v_k - x^*\|^2] +\frac{\mu\alpha_k}{4} \mathbb{E}[\|x_k - x^*\|^2] + \left(1 - \frac{\mu\alpha_k}{4}\right)\frac{C_{\beta,c,B_h}}{2}\mathbb{E}[\dist^2(x_{k}, \mathcal{X})] \\
		& \qquad\qquad\qquad \leq \left(1 - \frac{\mu\alpha_k}{4}\right)\mathbb{E}[\|v_{k-1} - x^*\|^2] +  \left( 1 + \frac{2}{C_{\beta,c,B_h} \left(1 - \frac{\mu \alpha_k}{4} \right)} \right) \alpha_{k}^2 B^2.\nonumber
	\end{align*}
Since $k > \blue{k_0 = \lfloor{\frac{8L}{\mu}} - 1\rfloor}$ and $\alpha_k = \frac{4}{\mu} \gamma_k$, we get: 
	\begin{align*}
		&\mathbb{E}[\|v_k - x^*\|^2] +\gamma_{k} \mathbb{E}[\|  x_k - x^*\|^2] + \left(1 - \gamma_{k}\right)\frac{C_{\beta,c,B_h}}{2} \mathbb{E}[\dist^2(x_{k}, \mathcal{X})]\\
		& \qquad\qquad\qquad \leq \left(1 - \gamma_{k}\right)\mathbb{E}[\|v_{k-1}- x^*\|^2] +\left( 1 + \frac{2}{C_{\beta,c,B_h} \left(1 - \gamma_k \right)} \right) \frac{16}{\mu^2}\gamma_{k}^2 B^2 \quad \forall  k> k_0.
	\end{align*}
    Note that in this case 	$\gamma_k = 2/(k+1)$ is a decreasing sequence, an thus we have:  
	\begin{align*}
		1 - \gamma_k = \frac{k-1}{k+1}\geq \frac{1}{3}  \quad \forall k \geq 2.
	\end{align*}
	Using this bound in the previous recurrence, we also get the second statement. 
\end{proof}

\noindent Now, we are ready to derive sublinear rates when we assume a \red{strongly convex condition} on the objective function.  Let us define for $k \geq k_0+1$ the sum: 
\begin{align*}
	\bar{S}_k = \sum_{j=k_0+1}^{k} (j+1)^2 \sim \mathcal{O} (k^3 - k_0^3), 
\end{align*}
and  the corresponding   average sequences: 
\begin{align*}
	&\hat{x}_k = \frac{\sum_{j=k_0+1}^{k} (j+1)^2 x_j}{\bar{S}_k}, \quad
    \text{and} \quad   \hat{w}_k = \frac{\sum_{j=k_0+1}^{k} (j+1)^2 \Pi_{\mathcal X}(x_j)}{\bar{S}_k} \in \mathcal{X}. 
\end{align*}

\begin{theorem}
\label{th:strcase2}
Let  $f(\cdot,\zeta), g(\cdot,\zeta)$ and  $h(\cdot,\xi)$  be  convex functions.  Additionally,  Assumptions \ref{assumption1}--\ref{assumption4} hold, with $\mu>0$.  Further, consider stepsize  $\alpha_k = \min\left(\frac{1}{L}, \frac{8}{\mu (k + 1)}\right)$ and $\beta \in \left(0,2\right)$. Then, for $k> k_0$, where \blue{$k_0 = \lfloor{\frac{8L}{\mu}} - 1\rfloor$}, we have the following sublinear convergence rates  for the  average sequence $	\hat{x}_k $ in terms of  optimality and feasibility violation for problem \eqref{eq:prob} (we keep only the dominant terms):
	\begin{align*}
		&\mathbb{E} [\|\hat{x}_k - x^*\|^2] \leq \mathcal{O} \left( \frac{B^2}{\mu^2 C_{\beta,c,B_h} (k+1)} \right),\\
		&\mathbb{E}\left[\dist^2(\hat{x}_k, \mathcal{X})\right]  \leq \mathcal{O}\left( \frac{B^2}{\mu^2C_{\beta,c,B_h}^2 (k + 1)^2}\right).
	\end{align*}
\end{theorem}

\begin{proof}
From Lemma \ref{lemma3.4}, the following recurrence is valid for   any $k>k_0$:  
	\begin{align*}
		&\mathbb{E}[\|v_k - x^*\|^2] +\gamma_{k} \mathbb{E}[\| x_k - x^*\|^2] + \frac{C_{\beta,c,B_h}}{6} \mathbb{E}[\dist^2(x_{k}, \mathcal{X})]\\
		& \qquad\qquad\qquad \leq \left(1 - \gamma_{k}\right)\mathbb{E}[\|v_{k-1}- x^*\|^2] + \left( 1 + \frac{6}{C_{\beta,c,B_h} } \right) \frac{16 B^2}{\mu^2}\gamma_{k}^2.
	\end{align*}
From definition of $\gamma_{k}=\frac{2}{k+1}$
and multiplying  the whole inequality with $(k+1)^2$, we get:
	\begin{align*}
		&(k+1)^2\mathbb{E}[\|v_k - x^*\|^2] +2(k+1) \mathbb{E}[\| x_k - x^*\|^2] + \frac{C_{\beta,c,B_h}}{6} (k+1)^2\mathbb{E}[\dist^2(x_{k}, \mathcal{X})]\\
		&  \leq k^2 \mathbb{E}[\|v_{k-1}- x^*\|^2] + \left( 1 + \frac{6}{C_{\beta,c,B_h} } \right) \frac{64}{\mu^2} B^2.
	\end{align*}
Summing  this inequality from $k_0+1$ to $k$, we get:
	\begin{align*}
		&{(k+1)^2}\mathbb{E}[\|v_k - x^*\|^2]+ 2 \sum_{j=k_0+1}^{k}(j+1) \mathbb{E}[\| x_j - x^* \|^2]\\
		& + \frac{C_{\beta,c,B_h}}{6} \sum_{j=k_0+1}^{k}(j+1)^2  \mathbb{E}[ \dist^2(x_j,\mathcal{X})] \\
		& \leq (k_0+1)^2 \mathbb{E}[\|v_{k_0} -x^*\|^2] + \left( 1 + \frac{6}{C_{\beta,c,B_h} } \right) \frac{64 B^2}{\mu^2} (k-k_0).
	\end{align*}
	By linearity of the expectation operator, we further have: 
	\begin{align}\label{expectLinear}
		&{(k+1)^2}\mathbb{E}[\|v_k - \bar{v}_{k}\|^2]+ \frac{2}{(k+1)}\mathbb{E} \left[\sum_{j=k_0+1}^{k}(j+1)^2 \| x_j - x^*\|^2\right]\nonumber\\
		& + \frac{C_{\beta,c,B_h}}{6}  \mathbb{E}\left[\sum_{j=k_0+1}^{k}(j+1)^2  \dist^2(x_j,\mathcal{X}) \right] \\
		& \leq (k_0+1)^2  \mathbb{E}[\|v_{k_0} -\bar{v}_{k_0}\|^2] + \left( 1 + \frac{6}{C_{\beta,c,B_h} } \right) \frac{64 B^2}{\mu^2} (k-k_0)\nonumber.
	\end{align}
and using convexity of the norm, we get: 
	\begin{align*}
		&{(k+1)^2}\mathbb{E}[\|v_k - x^*\|^2]+ \frac{2 \bar{S}_k}{(k+1)}\mathbb{E} [\|\hat{x}_k- x^* \|^2] + \frac{\bar{S}_k C_{\beta,c,B_h}}{6}  \mathbb{E}[\|\hat{w}_k-\hat{x}_k\|^2]  \nonumber\\
		& \leq (k_0+1)^2 \mathbb{E}[\|v_{k_0} -x^* \|^2] + \left( 1 + \frac{6}{C_{\beta,c,B_h} } \right) \frac{64 B^2}{\mu^2} (k-k_0), 
	\end{align*}
After some simple calculations and keeping only the dominant terms, we get the following convergence rate for the average sequence $\hat{x}_k$ in terms of optimality:
	\begin{align*}
		&\mathbb{E} [\|\hat{x}_k - x^*\|^2] \leq \mathcal{O} \left( \frac{B^2}{\mu^2 C_{\beta,c,B_h} (k+1)} \right), \\
		&\mathbb{E}[\|\hat{w}_k-\hat{x}_k\|^2] \leq \mathcal{O}\left( \frac{B^2}{\mu^2C_{\beta,c,B_h}^2 (k+1)^2}\right).
	\end{align*}
Since $\hat{w}_k \in \mathcal{X}$,  we get the following convergence rate for the average sequence $\hat{x}_k$ in terms of  feasibility violation: 
	\begin{align*}
		\mathbb{E}[\dist^2(\hat{x}_k,\mathcal{X})] &\leq \mathbb{E}[\|\hat{w}_k-\hat{x}_k\|^2] \leq \mathcal{O}\left( \frac{B^2}{\mu^2C_{\beta,c,B_h}^2 (k + 1)^2}\right). 
	\end{align*}
This proves our statements. 
\end{proof}

\noindent Recall the expression of $C_{\beta,c,B_h}$:
\begin{align*}
	C_{\beta,c,B_h}= \left(\frac{\beta(2-\beta)}{cB_h^2}\right)\left(1-\frac{\beta(2-\beta)}{cB_h^2}\right)^{-1} 
= \left(\frac{\beta(2-\beta)}{cB_h^2 - \beta(2-\beta)}\right).
\end{align*}
For the particular choice of the stepsize  $\beta =1$, we have:
\begin{align*}
	C_{1,c,B_h}= \left(\frac{1}{cB_h^2 - 1}\right) > 0,
\end{align*}
since  we always have $cB_h^2 >1$.  Using this expression  in the convergence rates of Theorem \ref{th:strcase2}, we obtain: 
	\begin{align*}
		&\mathbb{E}\left[\dist^2(\hat{x}_k, \mathcal{X})\right]  \leq \mathcal{O}\left(\frac{B^2 (cB_h^2 - 1)^2 }{\mu^2 (k + 1)^2}  \right),\\
		&\mathbb{E}\left[\|\hat{x}_k - x^*\|^2 \right]  \leq \mathcal{O}\left(  \frac{B^2(cB_h^2 - 1)}{\mu^2 (k+1)}\right).
	\end{align*}

\noindent We can easily notice from Lemma \ref{lemma3.4} that for $B=0$ we can get better convergence rates. More specifically, for this particular case, taking constant stepsize, we get linear rates for the last iterate $x_k$ in terms of optimality and feasibility violation. We state this result in the next corollary.
\begin{corollary}
Under the assumptions of Theorem \ref{th:strcase2}, with $B=0$, the last iterate $x_k$  generated by SSP algorithm with constant stepsize $\alpha_k \equiv  \alpha < \min(1/L, 4/\mu)$ converges linearly in terms of optimality and feasibility violation.
\end{corollary}	

\begin{proof}
\red{When  $B=0$ and the stepsize satisfies $\alpha_k = \alpha < \min(1/L, 4/\mu)$, from \eqref{eq:1way},  we obtain:
	\begin{align*}
		\mathbb{E}[\|v_{k} - x^*\|^2]&  + \frac{\mu \alpha}{4} \mathbb{E}[\| x_k - x^*\|^2] + \frac{1}{2} C_{\beta,c,B_h}\left(1 - \frac{\mu \alpha}{4} \right) \mathbb{E}[\dist^2(x_{k}, \mathcal{X})] \\
		& \le \left(1 - \frac{\mu \alpha}{4} \right) \mathbb{E}[\| v_{k-1} - x^*\|^2].
	\end{align*}
}	
Then, since  $1 - \mu \alpha/4 \in (0,1)$, we get immediately that:
\begin{align*}
	& \mathbb{E}[\|x_{k}  - x^*\|^2]  \leq \left(  1 - \frac{\mu \alpha}{4}  \right)^k \frac{4}{\mu \alpha}\|v_{0} - x^*\|^2,\\
	& \frac{1}{2} C_{\beta,c,B_h}\mathbb{E}[\dist^2(x_{k}, \mathcal{X})] \leq  \left(  1 - \frac{\mu \alpha}{4}  \right)^{k-1} \|v_{0} - x^*\|^2, 	
\end{align*}
which proves our statements. 
\end{proof}

\noindent  Note that  in \citep{Nec:20} it has been proved   that stochastic first order methods are converging linearly on optimization problems of the form \eqref{eq:prob}  without functional constraints and satisfying Assumption \ref{assumption1} with $B=0$.  This paper extends this result to a stochastic subgradient projection method on optimization problems with functional constraints \eqref{eq:prob} satisfying Assumption \ref{assumption1} with $B=0$.  To the best of our  knowledge,  these convergence rates are new  for stochastic subgradient  projection methods applied on the general class of optimization problems \eqref{eq:prob}.   In Section 4 we provide an example of an  optimization problem with functional constraints, that is the  constrained linear least-squares,  which satisfies   Assumption \ref{assumption1} with $B=0$.

\section{Stochastic subgradient  for constrained least-squares}
In this section we consider the problem of finding a solution  to a system of  linear equalities and inequalities, see also equation $(11)$ in \citep{LevLew:10}: 
\begin{align}\label{P:least-square}
	 \text{find}  \;  x \in \mathcal{Y} : \;  Ax = b,  \; Cx \leq d,
\end{align}  
where $A  \in \rset^{m \times n}$, $C \in \rset^{p \times n}$ and $\mathcal{Y}$ is a simple polyhedral set.  We assume that this system is consistent, i.e. it has at least one solution.  This problem can be reformulated equivalently  as a particular case of the optimization problem with functional constraints  \eqref{eq:prob}: 
\begin{align}
	\label{prob2}
	 & \min_{x\in \mathcal{Y}}\; f(x) \;\; \left( :=\frac{1}{2} \mathbb{E} \left[\|A_\zeta^T x - b_\zeta\|^2\right] \right )\\
	& \text{subject to} \;\;  C_\xi^T x - d_\xi \leq 0  \quad  \forall \xi \in \Omega_2,  \nonumber
\end{align}
where  $A_\zeta^T$ and $C_\xi^T$ are (block) rows partitions of matrices $A$ and $C$, respectively.  Clearly,  problem (\ref{prob2}) is a particular case of problem \eqref{eq:prob}, with $f(x,\zeta) = \frac{1}{2}\|A_\zeta^T x - b_\zeta\|^2$, $g(x,\zeta) = 0$, and $h(x,\xi) = C_\xi^T x - d_\xi$ (provided that $C_{\xi}$ is a row of $C$).  Let us  define the polyhedral subset partitions $\mathcal{C}_{\xi} = \{ x \in \rset^n: C_\xi^T x - d_\xi \leq 0 \}$ and $\mathcal{A}_{\zeta} = \{ x \in \rset^n: A_\zeta^T x - b_\zeta = 0 \}$. In this case the feasible set  is the  polyhedron  $\mathcal{X} =\{x \in \mathcal{Y}: \; C x \leq  d \}$ and the optimal set  is the  polyhedron: 
$$\mathcal{X}^* = \{x \in \mathcal{Y}:  \;  Ax=b, \;  Cx \leq d \}  = \mathcal{Y} \cap_{\zeta \in \Omega_1}  \mathcal{A}_{\zeta}   \cap_{\xi \in \Omega_2}  \mathcal{C}_{\xi}.  $$ Note that for the particular problem \eqref{prob2} Assumption \ref{assumption1} holds with $B=0$ and e.g.  $L=2 \max_\zeta \|  A_\zeta\|^2$, since $f^*=0$ and  we have:    
\[    \mathbb{E} [ \| \nabla f(x, \zeta) \|^2  ]  =   \mathbb{E} [ \| A_\zeta  (A_\zeta^T x -  b_\zeta) \|^2  ]   \leq 
 ( 2 \max_\zeta \|  A_\zeta\|^2 ) \left(\frac{1}{2} \mathbb{E} \left[\|A_\zeta^T x - b_\zeta\|^2\right] \right) = L \,  f(x). \]
It is also obvious that Assumption \ref{assumption3} holds, since the functional constraints are linear. Moreover,   for the constrained  least-squares problem,  we replace Assumptions \ref{assumption2} and  \ref{assumption4} with the well-known Hoffman  property  of a polyhedral set, see  Example 3 from Section 2 and also   \citep{PenVer:18,LevLew:10,NecNes:15}:
\begin{align}
	\label{hoffman}
	\dist^2(u,  \mathcal{X}^*)  \leq  c \cdot  \mathbb{E} \left[  	\dist^2(u,  \mathcal{A}_{\zeta}) +   	\dist^2(u,  \mathcal{C}_{\xi})  \right]  \quad \forall u \in \mathcal{Y}, 
\end{align}
for some $c  \in (0, \infty)$.    Recall that  the Hoffman  condition \eqref{hoffman}  always holds  for \textit{nonempty} polyhedral sets  \citep{PenVer:18}.
 For the constrained least-squares problem the SSP algorithm  becomes:
\begin{center}
	\noindent\fbox{%
		\parbox{11cm}{%
			\textbf{Algorithm 2 (SSP-LS)}:\\
			$\text{Choose} \; x_0 \in \mathcal{Y}, \; \text{stepsizes} \; \alpha_k>0 $  and $\beta \in (0, 2)$\\
			$\text{For} \; k \geq 0 \;  \text{repeat:}$
			\begin{align*}
				& \text{Sample independently} \;   \zeta_k \sim \textbf{P}_1 \;\text{and}\; \xi_k\sim \textbf{P}_2  \; \text{and update:}  \\
				& v_k = x_k - \alpha_k A_{\zeta_{k}} (A_{\zeta_{k}}^T x_{k} - b_{\zeta_{k}}) \\
				& z_k = (1-\beta)v_k + \beta \Pi_{\mathcal{C}_{\xi_k}}(v_k)\\
				& x_{k+1} = \Pi_{\mathcal{Y}}(z_k).
			\end{align*}
		}%
	}
\end{center}

\noindent  Note that  the update for $v_k$ can be written as  step \eqref{eq:algstep1} in SSP for $f(x,\zeta) = \frac{1}{2}\|A_\zeta^T x - b_\zeta\|^2$ and $g=0$.  In contrast to the previous section however, here  we  consider an adaptive  stepsize:
\begin{align*}
	\alpha_k = \delta \frac{\|A_{\zeta_{k}}^T x_{k} - b_{\zeta_{k}}\|^2}{\|A_{\zeta_{k}}(A_{\zeta_{k}}^T x_{k} - b_{\zeta_{k}})\|^2}, \;\; \text{where} \;\; \delta \in (0,2).
\end{align*}
Note that when $C_{\xi}$ is a row of $C$, then $z_k$ has the explicit expression:
\[ z_k = v_k - \beta \frac{(C_{\xi_{k}}^T v_k - d_{\xi_{k}})_+}{\|C_{\xi_{k}}\|^2} C_{\xi_{k}}, \]
which coincides with step \eqref{eq:algstep2} in SSP for $h(x,\xi) = C_\xi^T x - d_\xi$.  
Note that we can use  e.g., probability distributions dependent on the (block) rows of  matrices $A$ and $C$:
\begin{align*}
	\textbf{P}_1(\zeta = \zeta_k) = \frac{\|A_{\zeta_k}\|^2_F}{\|{A}\|_F^2}\;\; \text{and}\;\; \textbf{P}_2(\xi = \xi_k) = \frac{\|C_{\xi_k}\|^2_F}{\|{C}\|_F^2},
\end{align*} 
where $\|\cdot\|_F$ denotes the Frobenius norm of a matrix.   Note that our algorithm SSP-LS is different from Algorithm $4.6$ in \citep{LevLew:10} through the  choice of the stepsize  $\alpha_k$, of  the sampling rules and of the update law for  $x_k$ and it is more general as it allows to work with block of rows of the matrices $A$ and $C$.  Moreover,  SSP-LS includes the classical Kaczmarz’s method when solving linear systems of equalities.   In the next section we derive linear convergence rates for SSP-LS algorithm, provided that the system of equalities and inequalities is consistent, i.e. $\mathcal{X}^*$ is  nonempty.


\subsection{Linear convergence}
\noindent In this section we prove linear convergence for the sequence generated by the SSP-LS algorithm for solving the constrained  least-squares problem \eqref{prob2}. 	Let us define  \textit{maximum block condition number} over all the submatrices $A_{\zeta}$:
\[  \kappa_{\text{block}}  = \max_{ \zeta \sim \textbf{P}_1}  \| A_{\zeta}^T \|   \cdot  \| (A_{\zeta}^T)^\dagger \|,  \]
where  $(A_{\zeta}^T)^\dagger$  denotes  the pseudoinverse of $A_{\zeta}^T$. Note that  if $A_{\zeta}^T$ has full  rank, then $(A_{\zeta}^T)^\dagger  = A_{\zeta} (A_{\zeta}^T A_{\zeta})^{-1} $.  Then, we have the following result. 

\begin{theorem}
	Assume that the polyhedral set $\mathcal{X}^* = \{x \in \mathcal{Y}:  \;  Ax=b, \;  Cx \leq d \}$  is nonempty. Then, we have the following linear rate of convergence for the sequence $x_k$ generated by the SSP-LS algorithm: 
	\begin{align*}
		&\mathbb{E}\left[{\dist^2(x_{k}, \mathcal{X}^*)}\right] \leq \left( 1 - \frac{1}{c} \min \left( \frac{\delta(2 - \delta)}{2  \kappa_{\text{block}}^2},   \frac{2 - \delta}{4 \delta}, \frac{\beta (2 - \beta)}{2} \right) \right)^k \dist^2(x_{0}, \mathcal{X}^*). 
	\end{align*}
\end{theorem}

\begin{proof}
 From the  updates of the  sequences   $x_{k+1}$, $z_k $ and $v_k$ in SSP-LS algorithm,  we have:
	\begin{align*}
		&  \|  x_{k+1} - \bar{x}_{k+1} \|^2	\leq  \|  x_{k+1} - \bar{x}_{k} \|^2 = \| \Pi_{\mathcal{Y}}(z_k) - \Pi_{\mathcal{Y}}(\bar{x}_k) \|^2 \leq \| z_k - \bar{x}_k \|^2 \\
		& =  \|  v_{k} - \bar{x}_{k} + \beta(\Pi_{\mathcal{C}_{\xi_k}}(v_k) - v_k)\|^2  \\
		& = \| v_k - \bar{x}_k \|^2 + \beta^2\| \Pi_{\mathcal{C}_{\xi_k}}(v_k) - v_k \|^2  + 2 \beta \langle v_k - \bar{x}_k , \Pi_{\mathcal{C}_{\xi_k}}(v_k) - v_k \rangle \\
		& =  \| x_k - \bar{x}_k \|^2 + \alpha_{k}^2 \|  A_{\zeta_{k}}(A_{\zeta_{k}}^T x_k - b_{\zeta_{k}}) \|^2  - 2 \alpha_k \langle x_k - \bar{x}_k ,  A_{\zeta_{k}}(A_{\zeta_{k}}^T x_k - b_{\zeta_{k}})  \rangle \\
		& \qquad + \beta^2\| \Pi_{\mathcal{C}_{\xi_k}}(v_k) - v_k \|^2  + 2 \beta \langle v_k - \bar{x}_k , \Pi_{\mathcal{C}_{\xi_k}}(v_k) - v_k \rangle. 
	\end{align*}		
	Using the definition of $\alpha_k$ and that  $A_{\zeta_{k}}^T(x_k - \bar{x}_k)= A_{\zeta_{k}}^Tx_k - b_{\zeta_{k}}$, we further get: 
	\begin{align*}
		\|  x_{k+1} - \bar{x}_{k+1} \|^2	& \leq  \|  x_{k} - \bar{x}_{k} \|^2   - \delta(2 - \delta)\frac{\|   A_{\zeta_{k}}^T x_k - b_{\zeta_{k}} \|^4} {\|  A_{\zeta_{k}}(A_{\zeta_{k}}^T x_k - b_{\zeta_{k}}) \|^2}  + \beta^2\| \Pi_{\mathcal{C}_{\xi_k}}(v_k) - v_k \|^2   \\
		& \quad + 2 \beta \langle v_k - \bar{x}_k , \Pi_{\mathcal{C}_{\xi_k}}(v_k) - v_k \rangle  \\ 
		& = \|  x_{k} - \bar{x}_{k} \|^2   - \delta(2 - \delta)\frac{\|   A_{\zeta_{k}}^T x_k - b_{\zeta_{k}} \|^4} {\|  A_{\zeta_{k}}(A_{\zeta_{k}}^T x_k - b_{\zeta_{k}}) \|^2} + \beta^2\| \Pi_{\mathcal{C}_{\xi_k}}(v_k) - v_k \|^2   \\
		& \quad  - 2 \beta \langle \Pi_{\mathcal{C}_{\xi_k}}(v_k) - v_k , \Pi_{\mathcal{C}_{\xi_k}}(v_k) - v_k \rangle +  2 \beta \langle   \Pi_{\mathcal{C}_{\xi_k}}(v_k) - \bar x_k, \Pi_{\mathcal{C}_{\xi_k}}(v_k) - v_k \rangle.
	\end{align*}	
    From the optimality condition of the projection  we always have $\langle  \Pi_{\mathcal{C}_{\xi_k}}(v_k) -z ,  \Pi_{\mathcal{C}_{\xi_k}}(v_k) - v_k  \rangle \leq 0 $  for all $ z \in \mathcal{C}_{\xi_k}$. Taking $z = \bar{x}_k \in \mathcal{X}^* \subseteq \mathcal{C}_{\xi_k}$ in the previous relation, we finally get:
    \begin{align}
        \label{eq:basic_LS}	
        & \|  x_{k+1} - \bar{x}_{k+1} \|^2  \\	
        & \leq \|  x_{k} - \bar{x}_{k} \|^2   - \delta(2 - \delta)\frac{\|   A_{\zeta_{k}}^T x_k - b_{\zeta_{k}} \|^4} {\|  A_{\zeta_{k}}(A_{\zeta_{k}}^T x_k - b_{\zeta_{k}}) \|^2}  -  \beta(2 - \beta) \|   \Pi_{\mathcal{C}_{\xi_k}} (v_k) - v_k\|^2.    \nonumber
	\end{align}
	From the definition of $v_k$ and $\alpha_k$, we have:
    \begin{align*}
	    v_k = x_k - \alpha_{k} A_{\zeta_{k}} (A_{\zeta_{k}}^T x_k - b_{\zeta_{k}})  & \iff \alpha_{k}^2 \|  A_{\zeta_{k}} (A_{\zeta_{k}}^T x_k - b_{\zeta_{k}}) \|^2 = \| v_k - x_k \|^2\\
	    & \iff \frac{\|   A_{\zeta_{k}}^T x_k - b_{\zeta_{k}} \|^4} {\|  A_{\zeta_{k}}(A_{\zeta_{k}}^T x_k - b_{\zeta_{k}}) \|^2} = \frac{1}{\delta^2}\| v_k - x_k \|^2 .
    \end{align*}
    Also, from the definition of $z_{k}$, we have:
    \begin{align}\label{def:zk}
    	\| \Pi_{\mathcal{C}_{\xi_k}} (v_k) - v_k\|^2 = \frac{1}{\beta^2} \| z_{k} - v_k \|^2 .
    \end{align}
    Now, replacing these two relations in (\ref{eq:basic_LS}), we get:
   \begin{align}\label{rec:LS}
	   & \|  x_{k+1} - \bar{x}_{k+1} \|^2 \nonumber\\	
	   & \quad \leq \|  x_{k} - \bar{x}_{k} \|^2   - \frac{\delta(2 - \delta)}{\delta^2} \|  v_k - x_k \|^2  -  \frac{\beta(2 - \beta)}{\beta^2} \|  z_{k} - v_k\|^2  \nonumber\\
	   & \quad \leq \|  x_{k} - \bar{x}_{k} \|^2   -   \frac{\delta(2 - \delta)}{2 \kappa_{\text{block}}^2}   \frac{ \kappa_{\text{block}}^2}{\delta^2} \|  v_k - x_k \|^2   \nonumber \\
	   & \qquad  - \min \left(    \frac{\delta(2 - \delta)}{4 \delta^2},  \frac{\beta(2 - \beta)}{2} \right) \left( 2 \|  v_k - x_k \|^2  +  \frac{2}{\beta^2} \|  z_{k} - v_k\|^2 \right).
   \end{align}
 First, let us consider the subset $\mathcal{C}_{\xi_{k}}$. Then, we have:
   \begin{align*}
   	\dist^2(x_k, \mathcal{C}_{\xi_{k}}) & = \| x_k - \Pi_{\mathcal{C}_{\xi_{k}}}(x_k) \|^2 \leq \| x_k - \Pi_{\mathcal{C}_{\xi_{k}}}(v_k) \|^2 \leq 2 \| x_k - v_k \|^2 + 2 \| v_k - \Pi_{\mathcal{C}_{\xi_{k}}}(v_k) \|^2 \\
   	& \overset{(\ref{def:zk})}{\leq} 2 \| x_k - v_k \|^2 + \frac{2}{\beta^2} \| v_k - z_{k} \|^2. 
   \end{align*}
Second, let us consider the subset  $\mathcal{A}_{\zeta_{k}}$.   Since the corresponding $A_{\zeta_{k}}^T$ represents a block of rows of matrix $A$, the update for $v_k$ in SSP-LS can be written as: 
\[ v_k = (1 - \delta) x_k + \delta T_{\zeta_{k}}(x_k), \]
where the operator  $T_{\zeta_{k}}$ is given by the following expression
\[ T_{\zeta_{k}}(x_k) = x_k - \frac{\| A_{\zeta_{k}}^T x_k - b_{\zeta_{k}} \|^2}{\| A_{\zeta_{k}}(A_{\zeta_{k}}^T x_k - b_{\zeta_{k}}) \|^2} A_{\zeta_{k}}(A_{\zeta_{k}}^T x_k - b_{\zeta_{k}}). \]
Further, the projection of $x_k$ onto the subset $\mathcal{A}_{\zeta_{k}}$ is,  see e.g., \citep{HorJon:12}:
 \[ \Pi_{\mathcal{A}_{\zeta_{k}}}(x_k) = x_k - (A_{\zeta}^T)^\dagger (A_{\zeta_{k}}^T x_k - b_{\zeta_{k}}). \] 
Hence, we have:
\begin{align*}
	\dist^2(x_k, \mathcal{A}_{\zeta_{k}}) & = \| x_k - \Pi_{\mathcal{A}_{\zeta_{k}}}(x_k) \|^2 = \| (A_{\zeta}^T)^\dagger (A_{\zeta_{k}}^T x_k - b_{\zeta_{k}}) \|^2 \\
	& \leq \| (A_{\zeta}^T)^\dagger \|^2 \| A_{\zeta_{k}}^T x_k - b_{\zeta_{k}} \|^2 \\
	& = \| (A_{\zeta}^T)^\dagger \|^2 \frac{\| A_{\zeta_{k}}^T x_k - b_{\zeta_{k}} \|^2}{\| A_{\zeta_{k}}(A_{\zeta_{k}}^T x_k - b_{\zeta_{k}}) \|^2}  \| A_{\zeta_{k}}(A_{\zeta_{k}}^T x_k - b_{\zeta_{k}}) \|^2 \\
	& \leq \|  (A_{\zeta_k}^T)^\dagger \|^2 \| A_{\zeta_{k}}^T \|^2 \frac{\| A_{\zeta_{k}}^T x_k - b_{\zeta_{k}} \|^4}{\| A_{\zeta_{k}}(A_{\zeta_{k}}^T x_k - b_{\zeta_{k}}) \|^2}\\
	& = \kappa_{\text{block}}^2 \| T_{\zeta_{k}}(x_k) - x_k \|^2 \\
	& = \frac{ \kappa_{\text{block}}^2 }{\delta^2} \|x_k - v_k\|^2.
\end{align*}
Using these two relations  in (\ref{rec:LS}), we finally get the following recurrence:
  \begin{align}
  	  \label{eq:ls_rec}
  & 	\| x_{k+1} - \bar{x}_{k+1} \|^2 \\ 
  & \leq \|  x_{k} - \bar{x}_{k} \|^2   - \min \left( \frac{\delta(2 - \delta)}{2\kappa_{\text{block}}^2},   \frac{2 - \delta}{4 \delta}, \frac{\beta (2 - \beta)}{2} \right)   \left(  \dist^2(x_k, \mathcal{A}_{\zeta_{k}})   + 	\dist^2(x_k, \mathcal{C}_{\xi_{k}})    \right). \nonumber 
  \end{align} 
  Now, taking conditional expectation w.r.t.  $\mathcal{F}_{[k-1]}$  in \eqref{eq:ls_rec}  and using Hoffman inequality 	\eqref{hoffman},  we obtain:
  \begin{align*}
  & \mathbb{E}_{ \zeta_{k}, \xi_{k}} [ 	\| x_{k+1} - \bar{x}_{k+1} \|^2  | \mathcal{F}_{[k-1]} ] \\  
  & \leq \|  x_{k} - \bar{x}_{k} \|^2   - \frac{1}{c} \min \left( \frac{\delta(2 - \delta)}{2 \kappa_{\text{block}}^2},   \frac{2 - \delta}{4 \delta}, \frac{\beta (2 - \beta)}{2} \right) \dist^2(x_k, \mathcal{X^*}) \\
  	& = \left( 1 - \frac{1}{c} \min \left( \frac{\delta(2 - \delta)}{2 \kappa_{\text{block}}^2},   \frac{2 - \delta}{4 \delta}, \frac{\beta (2 - \beta)}{2} \right) \right) \|  x_{k} - \bar{x}_{k} \|^2.
  \end{align*}
Finally, taking full expectation,   recursively we get the statement of the theorem. 
\end{proof}
Note that for $\delta=\beta=1$ and $A_\zeta^T$   a single row of matrix $A$,   we have   $\kappa_{\text{block}} =1$ and we get a simplified estimate for linear convergence $\mathbb{E}\left[{\dist^2(x_{k}, \mathcal{X}^*)}\right] \leq  (1-1/(4c))^k \dist^2(x_{0}, \mathcal{X}^*)$, which is similar to convergence estimate for the algorithm in \citep{LevLew:10}.  In the block case,  for $\delta=\beta=1$ and assuming that   $\kappa_{\text{block}} \geq 2$, we get the linear convergence $\mathbb{E}\left[{\dist^2(x_{k}, \mathcal{X}^*)}\right] \leq  \left(1-1/(2c \kappa_{\text{block}}^2 ) \right)^k \dist^2(x_{0}, \mathcal{X}^*)$, i.e., our rate depends explicitly  on the geometric properties of the submatrices $A_\zeta$ and $C_\zeta$ (recall that both constants $c$ and $ \kappa_{\text{block}}$ are defined in terms of these submatrices).    From our best knowledge, this is the first time when such convergence bounds are obtained for a stochastic subgradient type algorithm solving constrained least-squares.


\section{Illustrative examples and numerical tests}
In this section, we present several applications where our algorithm can be applied,
such as the robust sparse SVM classification problem \citep{BhaGra:04},  sparse SVM  classification problem \citep{WesSch:03}, constrained least-squares and  linear programs \citep{Tib:11},  accompanied by detailed numerical simulations.  The codes were written
in Matlab and run on a PC with i7 CPU at 2.1 GHz and 16 GB memory.


\subsection{Robust sparse SVM classifier} \label{sec:5.1}
We consider a two class  dataset $\{(z_i,y_i)\}_{i=1}^N$, where $z_i$ is the vector of features and $y_i \in \{-1, 1\}$ is the corresponding label. A robust classifier is a hyperplane  parameterized by a weight vector $w$ and  an offset from the origin $d$, in which the decision boundary and set of relevant features are resilient to uncertainty in the data, see  equation $(2)$ in \citep{BhaGra:04} for more details.   Then, the robust sparse classification problem can be formulated as:
\begin{align*}
	& \min_{w,d,u}  \;    \lambda \sum_{i=1}^N u_i +  \|w\|_1  \\ 
	& \text{subject to}:  \;  y_i (w^T \bar{z}_i + d) \geq 1 -u_i  \quad \forall \bar{z}_i \in {\cal Z}_i, \;\; u_i \geq 0 \quad  \forall i=1:N,
\end{align*} 
where ${\cal Z}_i$ is the uncertainty set in the data $z_i$, the parameter $\lambda>0$ and $1$-norm is added in the objective to induce sparsity in $w$ and \blue{$u_i$'s are slack variables that provide a mechanism for handling an error in the assigned class }.    To find a  hyperplane that is robust and generalized well, each ${\cal Z}_i$ would need to be specified  by a large corpus of pseudopoints. In particular, finding a robust  hyperplane can be simplified  by considering a data uncertainty model in the form of  ellipsoids.  In this case we can convert infinite number of linear constraints into a single non-linear constraint and thus recasting the above  set of robust linear inequalities  as second order cone constraints. Specifically, if the uncertainty set for $i$th data is defined by an ellipsoid with the center $z_i$ and the shape given by the positive semidefinite matrix $Q_i \succeq 0$, i.e. ${\cal Z}_i = \{ \bar{z}_i:  \langle Q_i(\bar{z}_i - z_i), \bar{z}_i - z_i \rangle \leq 1  \}$, then a solution to the robust hyperplane  classification  problem is one in which \blue{the hyperplane parameterized  by $(w,d)$ does not intersect any ellipsoidal data uncertainty model (see Appendix 1 for proof):}
\begin{align}   
y_i(w^T z_i + d) \geq \|Q_i^{-1/2} w\| \blue{ - u_i}.  
\label{eq:socc}
\end{align}
Hence,   the robust classification problem can be recast as a convex optimization problem with many functional constraints that are either linear or second order cone constraints:
\begin{align*}
	& \min_{w,d,u}  \; \lambda \sum_{i=1}^N u_i +  \|w\|_1  \\
	& \text{subject to}:   \;\;  y_i (w^T z_i + d) \geq 1 -u_i,  \;\;  \; u_i \geq 0 \quad   \forall i=1:N\\
	& \qquad  \qquad \quad \;\; y_i(w^T z_i + d) \geq \|Q_i^{-1/2} w\| \blue { - u_i} \;\; \forall i=1:N.
\end{align*}

\noindent This is a particular form of problem \eqref{eq:prob} and thus we can solve it using our algorithm SSP.   Since every one of the $N$ data points has its own covariance matrix $Q_i$, this formulation results in a large optimization problem so it is necessary to impose some restrictions on the shape of these matrices. Hence, we consider two scenarios: (i) class-dependent covariance matrices, i.e., $Q_i = Q_+$ if $y_i = +1$ or $Q_i = Q_-$ if $y_i = -1$; (ii) class-independent covariance matrix, i.e., $Q_i =Q_{\pm}$ for all $y_i$. For more details on the choice of  covariance matrices see \citep{BhaGra:04}. Here, each covariance matrix is assumed to be diagonal. In a class-dependent diagonal covariance matrix, the diagonal elements of $Q_+$ or $Q_-$ are unique, while in class-independent covariance matrix, all diagonal elements of $Q_{\pm}$ are identical. Computational experiments designed to evaluate the performance of SSP require datasets in which the level of variability associated with the data can be quantified. Here, a noise level parameter, $0\leq \rho \leq 1$, is introduced to scale each diagonal element of the covariance matrix, i.e., $\rho Q_+$ or $\rho Q_-$ or $\rho Q_{\pm}$.  When $\rho = 0$, data points are associated with no noise (the nominal case). The $\rho$ value acts as a proxy for data variability.  For classifying a point we consider the following rules. The “ordinary rule” for classifying a data point $z$ is as follows: if $w^T_* z + d_* >0$, then $z$ is assigned to the $+1$ class;  if $w^T_* z + d_* <0$, then $z$ is identified with the $-1$ class. An ordinary error occurs when the class predicted by the hyperplane differs from the known class of the data point. The “worst case rule” determines whether an ellipsoid with center $z$ intersects the hyperplane. Hence, some allowable values of $z$ will be classified incorrectly if $| w^T_* z  + d_* | < \| Q_i^{-1/2} w_* \|^2$ (worst case error).

\begin{table}[ht]
	\caption{Comparison between nominal and robust classifiers on training data: class-dependent  covariance matrix (first half), class-independent covariance matrix (second half).}
	\centering
	\label{Robust_SVM}
	\begin{tabular}{|c|c|c|c|c|c|c|}
		\hline
		\multirow{2}{*}{$\lambda$} & \multirow{2}{*}{$\rho$} & \multicolumn{3}{c|}{\textbf{Robust}}                                       & \multicolumn{2}{c|}{\textbf{Nominal}}                         \\ \cline{3-7} 
		&                                      & \textbf{$w_* \neq 0$} & worst case error & ordinary  error & \textbf{$w_* \neq 0$} & ordinary error \\ \hline
		0.1                         & \multirow{4}{*}{0.01}                & 1699                                          &      3      &       332    & 546                                           &    198       \\ \cline{1-1} \cline{3-7} 
		0.2                         &                                      & 406                                          &      25      &     116  & 767                                           &         113      \\ \cline{1-1} \cline{3-7} 
		0.3                         &                                      & 698                                           &      14      &     111       & 774                                           &         67     \\ \cline{1-1} \cline{3-7} \hline
		0.1                         & \multirow{4}{*}{0.3}                & 1581                                          &      63      &       331    & 546                                           &    198       \\ \cline{1-1} \cline{3-7} 
		0.2                         &                                      & 1935                                          &     86      &     326  & 767                                           &         113      \\ \cline{1-1} \cline{3-7} 
		0.3                         &                                      & 1963                                           &    77       &     311       & 774                                           &         67     \\ \cline{1-1} \cline{3-7} \hline\hline 
		0.1                         & \multirow{4}{*}{0.01}                & 1734                                          &      0      &       331    & 546                                           &    198       \\ \cline{1-1} \cline{3-7} 
		0.2                         &                                      & 1822                                           &      1      &     268  & 767                                           &         113      \\ \cline{1-1} \cline{3-7} 
		0.3                         &                                      & 1937                                           &      0      &     266       & 774                                           &         67     \\ \cline{1-1} \cline{3-7} \hline
		0.1                         & \multirow{4}{*}{0.3}                & 1629                                          &      19      &       316    & 546                                           &    198       \\ \cline{1-1} \cline{3-7} 
		0.2                         &                                      & 1899                                           &      19      &     296  & 767                                           &         113      \\ \cline{1-1} \cline{3-7} 
		0.3                         &                                      & 2050                                           &      20      &     282       & 774                                           &         67     \\ \cline{1-1} \cline{3-7} \hline
	\end{tabular}
\end{table}

\begin{table}[ht]
	\caption{Comparison between nominal and robust classifiers on testing data:  class-dependent  covariance matrix (first half), class-independent covariance matrix (second half). }
	\centering	
	\label{Robust_SVM2}
	\begin{tabular}{|c|c|c|c|}
		\hline
		\multirow{2}{*}{$\lambda$} & \multirow{2}{*}{$\rho$} & \textbf{Robust}   & \textbf{Nominal}  \\ \cline{3-4} 
		&                             & \text{accuracy (ordinary rule)} & \text{accuracy (ordinary rule)} \\ \hline
		0.1                         & \multirow{3}{*}{0.01}       &  180, 72.5\%               & 166,  66.9\%                \\ \cline{1-1} \cline{3-4} 
		0.2                         &                             &         200, 80.6\%       & 198, 79.8\%               \\ \cline{1-1} \cline{3-4} 
		0.3                         &                             &       198, 79.8\%        & 193, 77.9\%               \\ \hline
		0.1                         & \multirow{3}{*}{0.3}        &      180, 72.5\%           & 168, 67.8\%          \\ \cline{1-1} \cline{3-4} 
		0.2                         &                             &      200, 80.6\%         & 174, 70.2\%              \\ \cline{1-1} \cline{3-4} 
		0.3                         &                             &      198, 79.8\%         & 172, 69.4\%                \\ \hline \hline
		0.1                         & \multirow{3}{*}{0.01}       &   180, 72.5\%              & 167, 67.4\%           \\ \cline{1-1} \cline{3-4} 
		0.2                         &                             &      200, 80.6\%      & 196, 79.1\%             \\ \cline{1-1} \cline{3-4} 
		0.3                         &                             &    198,  79.8\%      &  193, 77.9\%                \\ \hline
		0.1                         & \multirow{3}{*}{0.3}        &      180, 72.5\%        & 165, 66.6\%          \\ \cline{1-1} \cline{3-4} 
		0.2                         &                             &        200, 80.6\%      & 183, 73.8\%              \\ \cline{1-1} \cline{3-4} 
		0.3                         &                             &      198, 79.8\%       & 159, 64.1\%               \\ \hline
	\end{tabular}
\end{table}

\medskip 

\noindent Tables 1 and 2 give the results of our  algorithm SSP for  robust ($\rho>0$) and nominal ($\rho=0$) classification formulations.  We choose the parameters $\lambda=0.1, 0.2, 0.3$, $\beta = 1.96$, and stopping criterion $10^{-2}$. We consider a dataset of CT scan images having two classes,  covid and non-covid, available at \url{https://www.kaggle.com/plameneduardo/sarscov2-ctscan-dataset}. This dataset contains CT scan images of dimension ranging from $190 \times 190$ to $410 \times 386$ pixels. To implement our algorithm we have taken $1488$ data in which $751$ are of Covid patients and $737$ of Non-Covid patients. Then, we divide them into training data and testing data. For training data we have taken $1240$ images in which $626$ are of Covid and $614$ are of Non-Covid. For testing data we have taken $248$ images in which $125$ are of Covid and $123$ of Non-Covid. We also resize all images into $190\times 190$ pixels. First half of the Tables \ref{Robust_SVM} and \ref{Robust_SVM2} correspond to feature-dependent covariance matrix and the second half to feature-independent covariance matrix. Table \ref{Robust_SVM} shows the results for the training data and Table \ref{Robust_SVM2}  for  the testing data.  As one can see from Tables \ref{Robust_SVM} and \ref{Robust_SVM2}, the robust classifier yields better accuracies on both training and testing datasets.


\subsection{Constrained least-squares}
\noindent Next, we consider  constrained least-squares problem  	\eqref{prob2}.  We compare the performance of our algorithm SSP-LS  and the algorithm in \citep{LevLew:10} on synthetic data matrices $A$ and $C$ generated from a normal distribution.  Both algorithms were stopped when $\max(\|Ax-b\|, \| (Cx-d)_+\|) \leq 10^{-3}$. The results for different sizes of matrices $A$ and $C$ are given in Table \ref{randn}. One can easily see from Table \ref{randn} the superior performance of our algorithm in both, number of full iterations (epochs, i.e. number of passes through data) and cpu time (in seconds). 

\begin{table}[ht]
	\caption{Comparison between SSP-LS  and algorithm in \citep{LevLew:10} in terms of epochs and cpu time (sec) on random least-squares problems.}
	\centering
	\label{randn}
	\begin{tabular}{|c | c | c | c | c | c | c | c|}
		\hline 
		\multirow{2}{*}{\textbf{$\delta=\beta$}}
		&\multirow{2}{*}{$m$}
		&\multirow{2}{*}{$p$}
		&\multirow{2}{*}{$n$}
		&\multicolumn{2}{c|}{SSP-LS} 
		&\multicolumn{2}{c|}{\citep{LevLew:10}}
		\\ \cline{5-8}
		& & & 
		& \text{epochs} & \text{cpu time (s)}
		& \text{epochs} & \text{cpu time (s)} 
		\\ \hline 
		0.96 & 900 & 900 & $10^3$ & \textbf{755} & \textbf{26.0} & 817 & 29.9 \\ \hline
		1.96 & 900 & 900 & $10^3$ & \textbf{591} & \textbf{20.1} & 787 & 26.2 \\ \hline
		0.96 & 900 & 1100 & $10^3$ & \textbf{624} & \textbf{23.2} & 721 & 23.9 \\ \hline
		1.96 & 900 & 1100 & $10^3$ & \textbf{424} & \textbf{16.7} & 778 & 27.3 \\ \hline
		0.96 & 9000 & 9000 & $10^4$ & \textbf{1688} & \textbf{5272.0} & 1700 & 5778.1 \\ \hline
		1.96 & 9000 & 9000 & $10^4$ & \textbf{1028} & \textbf{3469.2} & 1662 & 5763.7 \\ \hline
		0.96 & 9000 & 11000 & $10^4$ & \textbf{1224} & \textbf{5716.0} & 1437 & 5984.8 \\ \hline
		1.96 & 9000 & 11000 & $10^4$ & \textbf{685} & \textbf{2693.7} & 1461 & 5575.3 \\ \hline
		0.96 & 900 & $10^5$ & $10^3$ & \textbf{5} & \textbf{105.0} & 9 & 163.9 \\ \hline
		1.96 & 900 & $10^5$ & $10^3$ & \textbf{4} & \textbf{77.7} & 9 & 163.9 \\ \hline
		0.96 & 9000 & $10^5$ & $10^4$ & \textbf{64} & \textbf{2054.5} & 214 & 5698.5 \\ \hline
		1.96 & 9000 & $10^5$ & $10^4$ & \textbf{42} & \textbf{1306.5} & 213 & 5156.7 \\ \hline
		0.96 & 9000 & 9000 & $10^5$ & \textbf{23} & \textbf{1820.3} & 23 & 1829.4 \\ \hline
		0.96 & 9000 & 11000 & $10^5$ & \textbf{21} & \textbf{2158.7} & 23 & 2193.1 \\ \hline
		1.96 & 9000 & 11000 & $10^5$ & \textbf{19} & \textbf{1939.7} & \text{23} & \text{2216.9} \\ \hline
		0.96 & 900 & 900 & $10^5$ & \textbf{14} & \textbf{65.7} & 17 & 86.8 \\ \hline
		1.96 & 900 & 1100 & $10^5$ & \textbf{13} & \textbf{74.3} & 15 & 75.6 \\ \hline
	\end{tabular}
\end{table}


\subsection{Linear programs}\label{sec5.3}
\noindent Next, we consider solving linear programs (LP) of the form:
\begin{align*}
	&\min_{\mathbf{z} \geq 0} \mathbf{c}^T \mathbf{z}   \quad \text{subject to} \quad  \mathbf{C}  \mathbf{z} \leq \mathbf{d}.  
\end{align*}
Using the primal-dual formulation this problem is equivalent to \eqref{P:least-square}: 
\[  \text{find} \; \mathbf{z} \in [ 0, \infty)^n, \mathbf{\nu}  \in [ 0, \infty)^p:   \mathbf{c}^T \mathbf{z}  + \mathbf{d}^T \mathbf{\nu} =0,  \;   \mathbf{C}  \mathbf{z} \leq \mathbf{d}, \;  \mathbf{C}^T  \mathbf{\nu} + \mathbf{c} \geq 0. \]
Therefore, we can easily identify in \eqref{P:least-square}: \[ x = \begin{bmatrix}
	\mathbf{z}\\
	\nu
\end{bmatrix} \in  \mathcal{Y} =  [ 0,\infty )^{n+p}, \;  {A} = [\mathbf{c}^T \; \mathbf{d}^T] \in \mathbb{R}^{n+p},  \; {C} = \begin{bmatrix}
 \mathbf{C}\qquad 0_{p \times p}\\
0_{n \times  n}\; - \mathbf{C}^T
\end{bmatrix}.  \]
Hence, we can use our algorithm SSP-LS to solve LPs. In Table \ref{LP} we compare the performance of our algorithm, the algorithm in \citep{LevLew:10} and Matlab solver \textit{lsqlin}  for solving the least-squares formulation of LPs taken from  the Netlib library  available on \url{https://www.netlib.org/lp/data/index.html}, and Matlab format LP library available on \url{https://users.clas.ufl.edu/hager/coap/Pages/matlabpage.html}. The first two algorithms were stopped when $\max(\|Ax-b\|, \| (Cx-d)_+\|) \leq 10^{-3}$ and we choose $\delta=\beta=1.96$.  In  Table \ref{LP}, in the first column after the name of the LP we provide the dimension of the matrix $\mathbf{C}$. From Table  \ref{LP}  we observe that SSP-LS is always better than the algorithm in  \citep{LevLew:10} and for large dimensions it also better than \textit{lsqlin}, a Matlab solver specially dedicated for solving constrained least-squares problems.  Moreover, for "qap15" dataset \textit{lsqlin} yields out of memory.
 
\begin{table}[ht]
	\caption{Comparison between SSP-LS, algorithm in \citep{LevLew:10} and Matlab solver lsqlin in terms of epochs and cpu time (sec) on  real data LPs. }
	\centering
	\label{LP}
	\begin{tabular}{|c|c|c|c|c|c|}
		\hline 
		\multirow{2}{*}{\text{LP}} & \multicolumn{2}{c|}{\text{SSP-LS}} & \multicolumn{2}{c|}{  \citep{LevLew:10} } & \text{lsqlin}   \\ \cline{2-6} 
		& \text{epochs}  & \text{time (s)}  & \text{epochs}     & \text{time (s)}     & \text{time (s)}    \\ \hline
		afiro (21$\times$51)                             & \textbf{1163}            & \textbf{1.9}            & 5943               & 2.7              & 0.09                \\ \hline
		beaconfd (173$\times$295)                         & \textbf{1234}             & \textbf{9.9}            &     9213           &     63.5              &      1.0              \\ \hline
		kb2   (43$\times$68)                       & \textbf{10}             & \textbf{0.02}            &    17                 &   0.03                &     0.14                \\ \hline
		sc50a (50$\times$78)                       & \textbf{9}             & \textbf{0.04}            &    879                 &   1.9                &     0.14               \\ \hline
		sc50b  (50$\times$78)                      & \textbf{25}             & \textbf{0.1}            &    411                 &   0.8                &     0.2             \\ \hline
		share2b (96$\times$162)                    & \textbf{332}             & \textbf{1.8}            & 1691               &   84.9                &     0.2              \\ \hline
		degen2      (444$\times$757)              & \textbf{4702}             & \textbf{380.6}      &  5872           &  440.6                 & 9.8               \\ \hline
		fffff800    (524$\times$1028)               & \textbf{44}             & \textbf{5.5}      &  80           &  9.3                 & 3.4                  \\ \hline
		israel   (174$\times$316)               & \textbf{526}             & \textbf{5.4 }     &  3729          & 312.9                 & 0.3              \\ \hline
		lpi bgdbg1  (348$\times$649)        &\textbf{ 476}             & \textbf{15.4}      &  9717           & 263.1                 & 0.5         \\ \hline
		osa 07 (1118$\times$25067)         & \textbf{148}             & \textbf{3169.8}      &  631           & 7169.7                 & 3437.1        \\ \hline
		qap15 (6330$\times$22275)         &    \textbf{70}         &  \textbf{2700.7}     &  373           &  7794.6               &    *  \\ \hline
        fit2p (3000$\times$13525)         &       \textbf{7}      &  \textbf{65.7}     &      458       &   2188.3              &    824.9  \\ \hline
        maros r7 (3137$\times$9408)         &       \textbf{635}      &  \textbf{2346.8}     &   1671          &  3816.2               &   3868.8   \\ \hline
        qap12 (3193$\times$8856)         &       \textbf{54}      &  \textbf{374.1}     &  339        &    1081.8         & 3998.4     \\ \hline
	\end{tabular}
\end{table}


\subsection{Sparse linear SVM}
Finally, we consider the sparse linear SVM classification problem:  
\begin{align*}
	& \min_{w,d,u} \; \lambda \sum_{i=1}^{n} u_i +\|w\|_1  \\
	& \text{subject to}:   \;\;  y_i (w^T z_i + d) \geq 1 -u_i,  \;\;  \; u_i \geq 0 \quad   \forall i=1:N.   
\end{align*}
This can be easily recast as an LP and consequently as a least-squares problem.  In  Table  \ref{LinearSVM} we report the results provided by our algorithms  SSP-LS for solving sparse linear SVM problems. Here, the ordinary error has the same meaning as in Section \ref{sec:5.1}.  We use several datasets:  covid dataset  from \url{www.kaggle.com/plameneduardo/sarscov2-ctscan-dataset};  PMU-UD, sobar-72 and divorce datasets available on \url{https://archive-beta.ics.uci.edu/ml/datasets}; and the rest from LIBSVM library  \url{https://www.csie.ntu.edu.tw/~cjlin/libsvmtools/datasets/}.  In the first column, the first argument  represents the number of features and the second argument represents the number of data.  We divided each dataset into $80\%$ for training  and $20\%$ for testing.  For the LP formulation we use the simple observation that any scalar  $u$ can be written as $u= u_{+} - u_{-}$, with $u_{+}, u_{-} \geq 0$.   We use the same  stopping criterion as in Section \ref{sec5.3}.  In Table \ref{LinearSVM} we  provide the number of relevant features, i.e., the number of nonzero elements of the optimal $w_*$, and the number of misclassified data (ordinary error) on  real training and testing datasets. 
\begin{table}[ht]
	\caption{Performance of sparse linear SVM classifier: ordinary error and sparsity of  $w_*$ on real training and testing datasets.}
	\centering
	\label{LinearSVM}
	\begin{tabular}{|c|c|c|c|c|}
		\hline 
		\text{Dataset}               & $\lambda$ & \textbf{$w_* \neq 0$}      & \text{ordinary error (train)} & \text{ordinary error (test)} \\ \hline
		\multirow{2}{*}{sobar-72 (38$\times$72)}     & 0.1        & 7/38     & 9/58                     & 3/14                    \\ \cline{2-5} 
		& 0.5        & 12/38    & 1/58                     & 2/14                    \\ \hline
		\multirow{2}{*}{breastcancer (18$\times$683)}   & 0.1        & 9/18     & 12/547                   & 13/136                  \\ \cline{2-5} 
		& 0.5        & 9/18     & 17/547                   & 7/136                   \\ \hline
		\multirow{2}{*}{divorce (108$\times$170)}       & 0.1        & 23/108   & 3/136                    & 0/34                    \\ \cline{2-5} 
		& 0.5        & 13/108   & 0/136                    & 1/34                    \\ \hline
		\multirow{2}{*}{caesarian (10$\times$80)}     & 0.1        & 0/10(NA) &  *                    &  *                   \\ \cline{2-5} 
		& 0.5        & 5/10     & 14/64                    & 9/16                    \\ \hline
		\multirow{2}{*}{cryotherapy (12$\times$90)}  & 0.1        & 3/12     & 8/72                     & 5/18                    \\ \cline{2-5} 
		& 0.5        & 6/12     & 6/72                     & 1/18                    \\ \hline
		\multirow{2}{*}{PMU-UD (19200$\times$1051)}        & 0.1        & 13/19200 & 0/841                    & 70/210                  \\ \cline{2-5} 
		& 0.5        & 37/19200 & 0/841                    & 47/210                  \\ \hline
		\multirow{2}{*}{Covid (20000$\times$2481)}         & 0.1        &   428/20000    &  146/1985                  &   101/496                \\ \cline{2-5} 
		& 0.5        &     752/20000     &142/1985               &       97/496            \\ \hline
		\multirow{2}{*}{Nomao (16$\times$34465)}         & 0.1        &  12/14       &    3462/27572                    &      856/6893      \\ \cline{2-5} 
		& 0.5        &   11/14     &   3564/27572           &    885/6893             \\ \hline
		\multirow{2}{*}{Training (60$\times$11055)}          & 0.1        &   31/60     &  630/8844                      &           152/2211              \\ \cline{2-5} 
		& 0.5        &   30/60     &       611/8844                   &        159/2211                 \\ \hline
		\multirow{2}{*}{leukemia (14258$\times$38)}          & 0.1        &   182/14258     &  9/31                      &           2/7              \\ \cline{2-5} 
		& 0.5        &  271/14258      &   9/31                       &               2/7          \\ \hline
		\multirow{2}{*}{mushrooms (42$\times$8124)}          & 0.1        &   30/42     &     426/6500           &   121/1624                  \\ \cline{2-5} 
		& 0.5        &    31/42    &  415/6500                         &    105/1624                    \\ \hline
		\multirow{2}{*}{ijcnn1 (28$\times$49990)}          & 0.1        &   11/28     &   3831/39992            &   1021/9998                   \\ \cline{2-5} 
		& 0.5        &   10/28     &          3860/39992              &   992/9998             \\ \hline
		\multirow{2}{*}{phishing (60$\times$11055)}          & 0.1        &    56/60    &     2953/8844       &  728/2211                 \\ \cline{2-5} 
		& 0.5        &   51/60     &      2929/8844            &    725/2211                  \\ \hline
	\end{tabular}
\end{table}


\section*{Appendix}
\textbf{1.} \textit{Proof of inequality \eqref{eq:socc}}. \blue{We want to find an  hyperplane  parameterized by  $(w,d)$ which does not intersect any ellipsoidal data uncertainty model. Thus, one can easily see that the relaxed (via slack variable $u_i$)  robust linear inequality}
\[  y_i (w^T \bar{z}_i + d) \geq \blue{ - u_i}  \quad \forall \bar{z}_i \in {\cal Z}_i    \]
over the ellipsoid with the center in $z_i$
$${\cal Z}_i = \{ \bar{z}_i:  \langle Q_i(\bar{z}_i - z_i), \bar{z}_i - z_i \rangle \leq 1  \}, $$
can be written as optimization problem whose minimum value must satisfy: 
\begin{align*}
\blue{ - u_i}  \leq  & \min_{\bar{z}_i} \; y_i(w^T \bar{z}_i + d)\\
	&\text{subject to:} \;  (\bar{z}_i - z_i)^T Q_i(\bar{z}_i - z_i) - 1 \leq 0.
\end{align*}
The corresponding dual problem is as follows:
\begin{align}\label{Dual}
	& \max_{\lambda\geq0} \min_{\bar{z}_i} y_i( w^T \bar{z}_i + d) + \lambda ((\bar{z}_i - z_i)^T Q_i(\bar{z}_i - z_i) - 1).
\end{align}
Minimizing the above problem with respect to $\bar{z}_i$, we obtain: 
\begin{align}
\label{zBar}
 y_i w + 2 \lambda Q_i(\bar{z}_i - z_i) = 0  \iff  \bar{z}_i = z_i -\frac{y_i}{2 \lambda} Q_i^{-1} w.
\end{align} 
By replacing this value of $\bar{z}_i$ into the dual problem (\ref{Dual}), we get:
\begin{align*}
	\max_{\lambda\geq 0} \left[  - \frac{y_i^2}{4 \lambda}w^T Q_i^{-1} w +  y_i(w^T z_i + d) - \lambda\right].
\end{align*}
For the dual optimal  solution 
\begin{align*}
\lambda^* = \frac{1}{2} \sqrt{y_i^2(w^T Q_i^{-1} w)}= \frac{1}{2} \sqrt{(w^T Q_i^{-1} w)},
\end{align*}
we get the primal optimal solution
\begin{align*}
	&\bar{z}_i^* = z_i - \frac{y_i Q_i^{-1} w}{\sqrt{w^T Q_i^{-1} w}},  
\end{align*}
and consequently the second order cone condition 
\begin{align*}
y_i \left(w^Tz_i + d\right) \geq \|Q_i^{-1/2}w\| \blue{ - u_i} .
\end{align*}


\section*{Acknowledgments}
The research leading to these results has received funding from:  NO Grants 2014–2021, under project ELO-Hyp, contract no. 24/2020; UEFISCDI PN-III-P4-PCE-2021-0720, under project L2O-MOC, nr.  70/2022.


\bibliographystyle{authordate1}

\begin{thebibliography}{99}
	
	\bibitem[Bauschke and  Borwein, 1996]{BauBor:96}
	H. Bauschke and J. Borwein, \textit{On projection algorithms for solving convex feasibility problems}, SIAM Review, 38(3): 367--376, 1996.
	
	\bibitem[Bertsekas, 2011]{Ber:11}
	D.P. Bertsekas,  {\em Incremental proximal methods for large scale convex optimization},  Mathematical Programming, 129(2): 163--195, 2011.
	
	\bibitem[Bianchi et al., 2019]{BiaHac:17}
	P. Bianchi, W. Hachem and A. Salim,  {\em A constant step forward-backward algorithm involving random maximal monotone operators}, Journal of Convex Analysis, 26(2): 397--436, 2019.
	
	\bibitem[Bhattacharyya et al., 2004]{BhaGra:04}
	C.  Bhattacharyya,  L.R. Grate,  M.I. Jordan, L. El Ghaoui and S. Mian, \emph{Robust sparse hyperplane classifiers: Application to uncertain molecular profiling data},  Journal of Computational Biology, 11(6): 1073--1089, 2004.
	
	\bibitem[Devolder et al., 2014]{DevGli:14}
	O. Devolder,  F.  Glineur and Yu. Nesterov,  \textit{First-order  methods of smooth convex optimization with inexact oracle}, Mathematical Programming, 146: 37--75, 2014.
	
	\bibitem[Duchi and  Singer, 2009]{DucSin:09}
	J. Duchi and  Y. Singer, \textit{Efficient online and batch learning using forward backward splitting}, Journal of Machine Learning 	Research, 10: 2899--2934, 2009.
	
	
	\bibitem[Garrigos et al., 2023]{GarGow:23}
	G. Garrigos and R.M. Gower, \textit{Handbook of convergence theorems for (stochastic) gradient methods}, arXiv:2301.11235v2, 2023. 
	
	\bibitem[Hardt et al., 2016]{HarSin:16}
M. Hardt, B. Recht and Y. Singer, \textit{Train faster, generalize better: stability of stochastic gradient descent}, International Conference on Machine Learning, 2016.
	
	\bibitem[Hermer et al., 2020]{HerStu:20}
N. Hermer, D.R. Luke and A. Sturm, \textit{Random function iterations for stochastic fixed point problems}, arXiv:2007.06479, 2020.
	
	\bibitem[Horn and Johnson, 2012]{HorJon:12}
	R. A. Horn and C.R. Johnson, \textit{Matrix Analysis}, Cambridge University Press, 2012.
	
	\bibitem[Kundu et al., 2018]{KunBac:18}
	A. Kundu, F. Bach and C. Bhattacharya, {\em Convex optimization over inter-section of simple sets: improved convergence rate guarantees via an exact penalty approach},  International Conference on Artificial Intelligence and Statistics, 2018.
	
	\bibitem[Lewis and Pang, 1998]{LewPan:98}
	A. Lewis and J.  Pang, \emph{Error bounds for convex inequality systems}, Generalized Convexity, Generalized Monotonicity (J. Crouzeix, J. Martinez-Legaz and M. Volle eds.), Cambridge University Press, 75--110, 1998.
	
	\bibitem[Leventhal and Lewis, 2010]{LevLew:10}
	D. Leventhal  and  A.S. Lewis. \emph{Randomized Methods for linear constraints: convergence rates and conditioning}, Mathematics of Operations Research, 35(3):  641--654, 2010.
	
	\bibitem[Moulines and Bach, 2011]{MouBac:11}
	E. Moulines and F. Bach, \textit{Non-asymptotic analysis of	 stochastic approximation algorithms for machine learning},  Advances in Neural Information Processing Systems Conf., 2011.
	
	\bibitem[Mordukhovich and Nam, 2005]{MorNam:05}
	B.S. Mordukhovich and N.M.  Nam, \textit{Subgradient of distance functions with applications to Lipschitzian stability}, Mathematical Programming, 104: 635--668, 2005.
	
	\bibitem[Necoara, 2021]{Nec:20}
	I. Necoara, \textit{General convergence analysis of stochastic first order methods for composite  optimization},  Journal of Optimization Theory and Applications, 189: 66--95  2021.
	
	\bibitem[Necoara et al.,  2019]{NecNes:15}
	I. Necoara, Yu. Nesterov and F. Glineur, \textit{Linear convergence	of first  order methods for non-strongly convex optimization}, Mathematical Programming, 175(1): 69--107, 2019. 
	
	
	\bibitem[Nedelcu et al., 2014]{NedNec:14}
	V. Nedelcu, I. Necoara and Q. Tran Dinh, \textit{Computational complexity of inexact gradient augmented Lagrangian methods:  application to constrained MPC}, SIAM Journal  on  Control and Optimization, 52(5): 3109--3134, 2014.
	
	\bibitem[Nemirovski and Yudin, 1983]{NemYud:83}
A. Nemirovski and D.B. Yudin, \textit{Problem complexity and method efficiency in optimization}, Wiley Interscience, 1983.
	
	\bibitem[Nemirovski et al., 2009]{NemJud:09}
	A. Nemirovski, A. Juditsky, G. Lan and A. Shapiro, \textit{Robust stochastic approximation approach to stochastic programming}, SIAM Journal  Optimization, 19(4): 1574--1609, 2009.
	
	\bibitem[Nesterov, 2018]{Nes:18}
	Yu. Nesterov, \textit{Lectures on Convex Optimization},  Springer Optimization and Its Applications, 137, 2018.
	
	\bibitem[Nedich, 2011]{Ned:11}
	A. Nedich,  \textit{Random algorithms for convex minimization problems}, Mathematical Programming, 129(2): 225--273, 2011.
	
	\bibitem[Nedich and  Necoara, 2019]{AngNec:19}
	A. Nedich and I. Necoara, \emph{Random minibatch subgradient algorithms for
		convex problems with functional constraints}, Applied Mathematics and Optimization, 8(3): 801--833, 2019.
	
	
	\bibitem[Necoara and Patrascu, 2018]{PatNec:18}
	A. Patrascu and  I. Necoara, \emph{On the convergence of inexact projection first order methods for convex minimization}, IEEE Transactions  Automatic Control, 63(10): 3317--3329, 2018.
	
	\bibitem[Patrascu and  Necoara, 2018]{PatNec:17}
	A. Patrascu  and I. Necoara, \textit{Nonasymptotic convergence of stochastic proximal point algorithms for constrained convex optimization}, Journal of Machine Learning Research, 18(198): 1--42, 2018.
	
	\bibitem[Pena et al., 2021]{PenVer:18}
	J. Pena, J. Vera and L. Zuluaga, \emph{New characterizations of Hoffman constants for systems of linear constraints},  Mathematical Programming,  187: 79--109, 2021. 
	
	\bibitem[Polyak, 1969]{Pol:69}
    B.T.  Polyak,  \emph{ Minimization of unsmooth functionals }, USSR Computational Mathematics and Mathematical Physics, 9 (3), 14-29, 1969.
	
	
    \bibitem[Polyak, 2001]{Pol:01}
    B.T.  Polyak,  \emph{ Random algorithms for solving convex inequalities}, Studies in  Computational Mathematics, 8: 409--422,  2001.
	
	\bibitem[Robbins and Monro, 1951]{RobMon:51}
H. Robbins and S. Monro, \textit{A Stochastic approximation method}, The Annals of Mathematical Statistics, 22(3): 400–407, 1951.
	
	\bibitem[Rockafellar and  Uryasev, 2000]{RocUry:00}
	R.T.  Rockafellar and S.P. Uryasev, \emph{Optimization of conditional value-at-risk}, Journal of Risk,  2: 21--41, 2000. 
	
	\bibitem[Rosasco et al., 2019]{RosVil:14}
	L. Rosasco, S. Villa and B.C. Vu, \textit{Convergence of stochastic	proximal  gradient algorithm},  Applied Mathematics and  Optimization, 82: 891--917 , 2019.
	
	\bibitem[Rasch and Chambolle, 2020]{RasCha:20}
J. Rasch and A. Chambolle, \textit{Inexact first-order primal–dual algorithms}, Computational Optimization and Applications, 76: 381–-430, 2020.
	
	\bibitem[Tibshirani, 2011]{Tib:11}
	R. Tibshirani,  \textit{The solution path of the generalized
		lasso}, Phd Thesis, Stanford Univ., 2011.
	
	\bibitem[Tran-Dinh et al., 2018]{TraFer:18}
	Q. Tran-Dinh, O. Fercoq, and V. Cevher,  {\em A smooth primal-dual optimization framework for nonsmooth composite convex minimization}, SIAM Journal on Optimization, 28(1): 96--134, 2018.
	
	\bibitem[Villa et al., 2014]{VilRos:14}
	S. Villa, L. Rosasco, S. Mosci and A. Verri,  \textit{Proximal methods for the latent group lasso penalty}, Computational Optimization and Applications, 58: 381--407, 2014.
	
	\bibitem[Vapnik, 1998]{Vap:98}
	V. Vapnik, \textit{Statistical learning theory}, John Wiley, 1998.
	
	\bibitem[Weston et al., 2003]{WesSch:03}
J. Weston, A. Elisseeff and B. Scholkopf, \textit{Use of the zero norm with linear models and kernel methods}, Journal of Machine Learning Research, 3: 1439--1461, 2003.
	
	\bibitem[Wang et al., 2015]{WanChe:15}
	M. Wang, Y. Chen, J. Liu and Y. Gu, {\em  Random multiconstraint projection: stochastic gradient methods for convex optimization with many constraints},  arXiv: 1511.03760, 2015.
	
	\bibitem[Xu, 2020]{Xu:18}
	Y. Xu,  \emph{Primal-dual stochastic gradient method for convex programs with many functional constraints},  SIAM Journal on Optimization, 30(2): 1664--1692,  2020. 
	
	\bibitem[Yang and  Lin, 2016]{YanLin:16}
	T. Yang and Q. Lin, \textit{RSG: Beating subgradient method without smoothness and strong convexity},  Journal of Machine Learning Research, 19(6): 1--33, 2018.
	
\end{thebibliography}

\end{document}